\documentclass[12pt]{amsart}

\usepackage{multido}
\usepackage{graphicx,float}
\usepackage{amssymb}
\usepackage{mathabx}
\usepackage{hyperref}
\usepackage{enumerate}

\numberwithin{equation}{section}
\numberwithin{figure}{section}
\allowdisplaybreaks
\theoremstyle{plain} \newtheorem{theorem}{Theorem}[section]
\theoremstyle{plain} \newtheorem{proposition}[theorem]{Proposition}
\theoremstyle{plain} \newtheorem{lemma}[theorem]{Lemma}
\theoremstyle{plain} \newtheorem{corollary}[theorem]{Corollary}
\theoremstyle{definition} \newtheorem{definition}[theorem]{Definition}
\theoremstyle{definition} 
\theoremstyle{remark} \newtheorem{remark}[theorem]{Remark}
\theoremstyle{remark} \newtheorem{example}[theorem]{Example}

\newcommand{\E}{\mathsf{E}}
\renewcommand{\P}{\mathsf{P}}
\newcommand{\Q}{\mathsf{Q}}

\newcommand\cE{{\mathcal E}}
\newcommand\cF{{\mathcal F}}
\newcommand\cL{{\mathcal L}}
\newcommand\cB{{\mathcal B}}

\newcommand\cH{{\mathcal H}}
\newcommand\cM{{\mathcal M}}

\newcommand\cI{{\mathcal I}}
\newcommand\cQ{{\mathcal Q}}

\newcommand{\XX}{\mathfrak{X}}
\newcommand{\R}{\mathbb{R}}

\newcommand{\Lnot}[1][0]{\cL^#1}

\newcommand{\slimsup}{s\text{-}\limsup}

\renewcommand{\subseteq}{\subset}

\newcommand{\bX}{{\mbf X}}

\newcommand{\bK}{{\mbf K}}

\renewcommand{\bX}{\bX}
\let\Xi\varXi
\renewcommand{\epsilon}{\varepsilon}
\newcommand{\eps}{\varepsilon}
\renewcommand{\bX}{X}

\renewcommand{\bK}{K}

\DeclareMathOperator{\Gr}{Gr}
\DeclareMathOperator{\cl}{cl}
\DeclareMathOperator{\co}{co}
\DeclareMathOperator{\cone}{cone}
\DeclareMathOperator{\cm}{\mathbf{m}}
\DeclareMathOperator{\CM}{\mathbf{M}}
\DeclareMathOperator{\Int}{int}
\DeclareMathOperator{\epi}{epi}
\newcommand{\esssup}[1][\cH]{\mathrm{ess\,sup}_{#1}}
\newcommand{\essinf}[1][\cH]{\mathrm{ess\,inf}_{#1}}

\newcommand{\one}{{\mathbf{1}}}
\newcommand{\eqdef}{=}

\usepackage{colortbl}

\newlength{\querylen}
\usepackage{fancybox}

\begin{document}

\title[Core and convex hulls]{Conditional cores and conditional convex hulls
  of random sets}

\author[E. Lepinette]{Emmanuel Lepinette}
\address{CEREMADE, Paris-Dauphine University\\
  Place du Mar\'echal De Lattre De Tassigny, \\
  75775 Paris cedex 16, France}
\email{emmanuel.lepinette@ceremade.dauphine.fr}

\author[I. Molchanov]{Ilya Molchanov}
\address{University of Bern \\
 Institute of Mathematical Statistics and Actuarial Science \\
 Sidlerstrasse 5 \\
 CH-3012 Bern \\
 SWITZERLAND}
\email{ilya.molchanov@stat.unibe.ch}

\thanks{IM supported by the Swiss National Science Foundation Grant
  200021-153597.} 



\date{\today}

\begin{abstract}
  We define two non-linear operations with random (not necessarily closed) sets
  in Banach space: the conditional core and the conditional convex
  hull. While the first is sublinear, the second one is superlinear
  (in the reverse set inclusion ordering). Furthermore, we introduce
  the generalised conditional expectation of random closed sets and
  show that it is sandwiched between the conditional core and the
  conditional convex hull. The results rely on measurability
  properties of not necessarily closed random sets considered from the
  point of view of the families of their selections. Furthermore, we
  develop analytical tools suitable to handle random convex (not
  necessarily compact) sets in Banach spaces; these tools are based on
  considering support functions as functions of random arguments. 
  The paper is motivated by applications to assessing
  multivariate risks in mathematical finance.
\end{abstract}

\maketitle

\section{Introduction}
\label{sec:introduction}

Each almost surely non-empty random closed set $X$ (see definition in
Section~\ref{sec:graph-meas-rand}) in a Banach space admits a
measurable \emph{selection}, that is, a random element $\xi$ that
almost surely belongs to $X$. Moreover, $X$ equals the closure of a
countable family of its selections, called a \emph{Castaing
  representation} of $X$, see \cite{mo1}. If at least one selection is
Bochner integrable, then $X$ becomes the closure of a countable family
of \emph{integrable} selections, and the (selection)
\emph{expectation} $\E X$ of $X$ is defined as the closure of the set
of expectations for its all integrable selections, see
\cite{hia:ume77,mo1}.

Almost surely deterministic selections (if they exist) constitute the
set of \emph{fixed points} of $X$; this set is always a subset of $\E
X$. The union of supports of all selections is a superset of $\E X$;
it can be regarded as the \emph{support} of $X$.

In this paper we work out conditional variants of these concepts which
also apply to not necessarily closed grah measurable random
sets. Given a probability space $(\Omega,\cF,\P)$ and a
sub-$\sigma$-algebra $\cH$ of $\cF$, we introduce the concept of the
\emph{conditional core} $\cm(X|\cH)$, which relies on considering
selections measurable with respect to $\cH$. If $\cH$ is trivial, then
the conditional core becomes the set of fixed points of $X$; it is
also related to the essential intersection considered in
\cite{hes:ser:choir14}. The conditional core corresponds to the
concept of the conditional essential supremum (infimum) for a family
of random variables, see \cite{bar:car:jen03}.

If $X$ is a.s. convex, its conditional core can be obtained by taking
the conditional essential infimum of its support function. Taking the
conditional essential maximum leads to the dual concept of the
\emph{conditional convex hull} $\CM(X|\cH)$. While the conditional core of
the sum of sets is a superset of the sum of their conditional cores,
the opposite inclusion holds for the conditional convex hull. In other
words, the conditional core and the conditional convex hull are
\emph{non-linear} set-valued expectations.

The conventional conditional (selection) expectation $\E(X|\cH)$ is a
well-known concept for integrable random closed sets, see
\cite{hes02,hia:ume77}. We introduce a \emph{generalised} conditional
expectation $\E^g(X|\cH)$ based on working with the set of generalised
conditional expectation of all selections and show how it relates to
the conventional one. In particular, it is shown that $\E^g(X|\cH)=X$
is $X$ is $\cH$-measurable, no matter if $X$ is integrable or not.

The presented results are motivated by applications in mathematical
finance, where multiasset portfolios are represented as sets and their
\emph{risks} are also set-valued, see
\cite{ham:hey10,ham:rud:yan13,cas:mol14}. Working in the dynamic
setting requires a better understanding of conditioning operation with
random sets, and its iterative properties. In this relation, the
conditional core provides a simple \emph{conditional risk measure}
which generalises the concept of the essential infimum for multiasset
portfolios. In order to make a parallel with classical financial
concept (where real-valued risk measures are sublinear), the sets are
ordered by \emph{reverse} inclusion, so that the conditional core is
sublinear and the conditional convex hull is superlinear.

Section~\ref{sec:decomp-meas-vers} introduces random sets, their
selections and treats various measurability issues, in particular, it
is shown that each closed set-valued map admits a measurable
version. A special attention is devoted to the decomposability and
infinite decomposability properties, which are the key concepts
suitable to relate families of random vectors and selections of random
sets.

Section~\ref{sec:random-convex-sets} develops various analytical tools
suitable to handle random \emph{convex} sets. Random \emph{compact}
convex sets in Euclidean space can be efficiently explored by passing
to their support functions, and the same tool works well for weakly
compact sets in separable Banach spaces. Otherwise (e.g., for
unbounded closed sets in Euclidean space), the support function is
only lower semicontinuous and may become a non-separable random
function on the dual space. For instance, the support function of a
random half-space in Euclidean space with an isotropic normal almost
surely vanishes on all deterministic arguments. We show that this
complication can be circumvented by viewing the support function as a
function applied to \emph{random} elements in the dual space. In
particular, we prove a random variant of the 
well-known result saying that a closed convex set is given
by intersection of a countable number of half-spaces.
It is also shown that random convex closed sets can be alternatively
described as measurable epigraphs of their support functions, thereby
extending the fact known in the Euclidean setting for compact random
sets to all random convex closed sets in separable Banach spaces.

Section~\ref{sec:conditional-core} introduces and elaborates the
properties of conditional cores of random sets. Given a
sub-$\sigma$-algebra $\cH$, the conditional core $\cm(X|\cH)$ is the
largest $\cH$-measurable random closed set contained in $X$. While we
work with random sets in Banach spaces, the conditional core may be
defined for random sets in general Polish spaces. In linear spaces,
the conditional core is subadditive for the reverse inclusion, that is
the core of the sum of two random sets is a superset of the sum of
their conditional cores. 

While the conditional core is the largest random set contained in the
given one and is measurable with respect to a sub-$\sigma$-algebra
$\cH$, the conditional convex hull $\CM(X|\cH)$ is a smallest
$\cH$-measurable random convex closed set containing $X$.
Section~\ref{sec:cond-conv-hull} establishes the existence of the
conditional convex hull. It is shown that the support function of the
conditional convex hull is given by the essential supremum of the
support function of $X$. Duality relationships between the core and
the convex hull are also obtained. 

Section~\ref{sec:expect-cond-expect} introduces the concept of a
generalised conditional expectation for random sets and shows that it
is sandwiched between the conditional core and the conditional convex
hull. By taking the intersection (or closed convex hull) of generalised
conditional expectations with respect to varying probability measures,
it is possible to obtain a rich collection of conditional set-valued
non-linear expectations. 

Random convex cones have a particular property that their support
functions either vanish or are infinite.
Conditional cores and convex hulls of random convex cones are
considered in Section~\ref{sec:random-cones}. 

Some useful facts about conditional essential supremum and infimum of
random variables and the generalised conditional expectation are
collected in the appendices.

\section{Decomposability and measurable versions}
\label{sec:decomp-meas-vers}

\subsection{Graph measurable random sets and their selections}
\label{sec:graph-meas-rand}

Let $\XX$ be a separable (real) Banach space with norm $\|\cdot\|$ and
the Borel $\sigma$-algebra $\cB(\XX)$ generated by its strong
topology. The norm-closure of a set $A\subset\XX$ is denoted by $\cl
A$ and the interior by $\Int A$.

Fix a complete probability space $(\Omega,\cF,\P)$.  Let $\cH$ be a
sub-$\sigma$-algebra of $\cF$, which may coincide with $\cF$. Denote
by $\Lnot[p](\XX,\cH)$ the family of $\cH$-measurable random elements
in $\XX$ with $p$-integrable norm for $p\in[1,\infty)$, essentially
bounded if $p=\infty$, and all random elements if $p=0$.
The closure in the strong topology in $\Lnot[p]$ for $p\in[1,\infty)$
is denoted by $\cl_p$ and $\cl_0$ is the closure in probability for
$p=0$. If $p=\infty$, the closure is considered in the
$\sigma(\Lnot[\infty],\Lnot[1])$-topology.

An \emph{$\cH$-measurable random set} (shortly, random set) is a
set-valued function $\omega\mapsto X(\omega)\subseteq \XX$ from
$\Omega$ to the family of all subsets of $\XX$, such that its graph
\begin{equation}
  \label{eq:5}
  \Gr X\eqdef\{(\omega,x)\in \Omega\times \XX:
  x\in X(\omega)\}
\end{equation}
belongs to the product $\sigma$-algebra $\cH\otimes \cB(\XX)$; in this
case $X$ is often called \emph{graph measurable}, see
\cite[Sec.~1.2.5]{mo1}. Unless otherwise stated, by the measurability we always
understand the measurability with respect to $\cF$. The random set $X$
is said to be closed (convex, open) if $X(\omega)$ is a closed
(convex, open) set for almost all $\omega$. If $X$ is closed, then
\eqref{eq:5} holds if and only if $X$ is \emph{Effros measurable},
that is $\{\omega:\; X(\omega)\cap G\neq\emptyset\}\in\cF$ for each
open set $G$, see \cite[Def.~1.2.1]{mo1}. 

\begin{definition}
  An $\cH$-measurable random element $\xi$ such that
  $\xi(\omega)\in X(\omega)$ for almost all $\omega\in\Omega$ is
  said to be an $\cH$-measurable \emph{selection} (selection in short)
  of $X$, $\Lnot(X,\cH)$ denotes the family of all
  $\cH$-measurable selections of $X$, and $\Lnot[p](X,\cH)$
  is the family of $p$-integrable ones for $p\in[1,\infty]$.
\end{definition}

It is known that each almost surely non-empty random set has at least
one selection, see \cite[Th.~4.4]{hes02}.

\begin{lemma} 
  \label{ApproxSelector}
  Let $\{\xi_n,n\geq1\}$ be a sequence from $\Lnot(\XX,\cF)$, so
  that $X(\omega)=\cl\{\xi_n(\omega),n\geq1\}$ is a random closed
  set.  Let $\xi\in \Lnot(X,\cF)$. Then, for each $\eps>0$, there
  exists a measurable partition $A_1,\dots,A_n$ of $\Omega$ such that
  \begin{displaymath}
    \E\Big[\big\| \xi-\sum_{i=1}^n \one_{A_i}\xi_i\big\|\wedge 1\Big]
    \le \epsilon.
  \end{displaymath}
\end{lemma}
\begin{proof}
  Consider a measurable countable partition
  $\{B_i, i\geq1\}$ of $\Omega$, such that $\|\xi-\xi_i\|<\epsilon/2$
  on $B_i$ and choose a large enough $n$, so that
  \begin{displaymath}
    \E\Big[\one_{\cup_{i\ge n+1} B_i}\|\xi-\xi_1\|\wedge 1\Big]\le \epsilon/2.
  \end{displaymath}
  Define $A_1\eqdef B_1\cup \left( \cup_{i\ge n+1} B_i\right)$ and
  $A_i\eqdef B_i$ for $i=2\,\dots,n$.  Since the mapping $x\mapsto x\wedge
  1$ is increasing,
  \begin{displaymath}
    \E\Big[\big\| \xi-\sum_{i=1}^n \one_{A_i}\xi_i\big\|\wedge 1\Big]\le
    \E\Big[\big(\sum_{i=1}^n\|\xi-\xi_i\|\one_{A_i}\big)\wedge 1\Big].
  \end{displaymath}
  Then 
  \begin{multline*}
    \E\Big[\big\| \xi-\sum_{i=1}^n \one_{A_i}\xi_i\big\|\wedge 1\Big]
    \le \sum_{i=1}^n \E\big[\|\xi-\xi_i\|\one_{A_i}\wedge 1\big]\\
    \le \sum_{i=1}^n \E\big[\|\xi-\xi_i\|\one_{B_i}\wedge 1\big]
    +\E\big[\one_{\cup_{i\ge n+1} B_i}\|\xi-\xi_1\|\wedge 1\big]\le
    \epsilon. \qquad \qedhere
  \end{multline*}
\end{proof}

\begin{definition}
  A family $\Xi\subset \Lnot(\XX,\cF)$ is said to be
  \emph{$\cH$-decomposable} if, for each
  $\xi,\eta\in\Xi$ and each $A\in \cH$, the random element
  $\one_A\xi+\one_{A^c}\eta$ belongs to $\Xi$.  
\end{definition}

Decomposable subsets of $\Lnot(\R^d,\cF)$ are studied under the
name stable sets in \cite{cher:kup:vog14}.  The following result is
well known for $p=1$ \cite{hia:ume77}, for $p\in[1,\infty]$
\cite[Th.2.1.6]{mo1}, and is mentioned in
\cite[Prop.~5.4.3]{kab:saf09} without proof for $p=0$. We give below
the proof in the latter case and provide its variant for random
\emph{convex} sets.

\begin{theorem} 
  \label{ReprDecompL0} 
  Let $\Xi$ be a non-empty subset of $\Lnot[p](\XX,\cF)$ for $p=0$ or
  $p\in[1,\infty]$. Then $\Xi=\Lnot[p](X,\cF)$ for a random closed set
  $X$ if and only if $\Xi$ is $\cF$-decomposable and closed. The
  family $\Xi$ is convex (is a cone in $\Lnot[p](\XX,\cF)$) if and
  only if $X$ is convex (is a cone in $\XX$).
\end{theorem}
\begin{proof}
  The necessity is trivial. Let $p=0$ and assume that $\Xi$ is
  $\cF$-decomposable and closed.  Consider a countable dense set
  $\{x_i,i\geq1\}\subset\XX$ and define
  \begin{displaymath}
    a_i\eqdef\inf_{\eta\in \Xi}\E\left[\|\eta-x_i\|\wedge 1\right],\quad i\geq1.
  \end{displaymath}
  For all $i,j\geq1$, there exists an $\eta_{ij}\in \Xi$ such that
  \begin{displaymath}
    \E\left[\|\eta_{ij}-x_i\|\wedge 1\right]\le a_i+j^{-1}.
  \end{displaymath}
  Define $X(\omega)\eqdef\cl\{\eta_{ij}(\omega),i,j\geq1\}$ for all
  $\omega\in \Omega$. Since $\Xi$ is decomposable and closed,
  $\Lnot(X,\cF)\subseteq \Xi$ by Lemma~\ref{ApproxSelector}. If $\Xi$
  is convex or is a cone, the same inclusion holds for $X$ being the
  closed convex hull of $\{\eta_{ij}(\omega),i,j\geq1\}$ or the closed
  cone generated by these random elements.

  Assume that there exists a $\xi\in \Xi$, which does not belong to
  $\Lnot(X,\cF)$. Then there exists a $\delta\in (0,1)$ such that
  \begin{displaymath}
    A=\bigcap_{i,j\geq1}\{\|\xi-\eta_{ij}\|\wedge 1>\delta\}
  \end{displaymath}
  has positive probability.  Since $\Omega=\cup_i\{\|\xi-x_i\|\wedge
  1<\delta/3\}$, the event $B_i\eqdef A\cap \{\|\xi-x_i\|\wedge
  1<\delta/3\}$ has a positive probability for some $i\geq1$.  Recall
  that
  \begin{displaymath}
    \|a-b\|\wedge 1\le \|a-c\|\wedge 1+\|c-b\|\wedge 1.
  \end{displaymath}
  Then, on the set $B_i$,
  \begin{displaymath}
    \|x_i-\eta_{ij}\|\wedge 1\ge \|\xi-\eta_{ij}\|
    \wedge 1-\|\xi-x_i\|\wedge 1\ge \frac{2\delta}{3}.
  \end{displaymath}
  Furthermore, $\eta'_{ij}\eqdef\xi\one_{B_i}+\eta_{ij}\one_{B^c_i}\in\Xi$
  by decomposability, and 
  \begin{multline*}
    j^{-1}\geq \E \big[ \|\eta_{ij}-x_i\|\wedge 1\big]-a_i
    \ge \E\big[\|\eta_{ij}-x_i\|\wedge 1\big]-\E\big[\|\eta'_{ij}-x_i\|\wedge 1\big]\\
    \ge \E\Big[\left(\|\eta_{ij}-x_i\|\wedge 1-\|\xi-x_i\|\wedge
      1\right)\one_{B_i}\Big]\ge \frac{\delta}{3}\P(B_i).
  \end{multline*}
  Since $B_i$ and $\delta$ do not depend on $j$, letting $j\to\infty$
  yields a contradiction.
\end{proof}

\begin{corollary}
  \label{cor:H-dec}
  If $\Xi\subset\Lnot[p](\XX,\cF)$ with $p=0$ or $p\in[1,\infty]$ is
  closed and $\cH$-decomposable, then there exists an $\cH$-measurable
  random closed set $X$ such that 
  \begin{displaymath}
    \Xi\cap\Lnot[p](\XX,\cH)=\Lnot[p](X,\cH).
  \end{displaymath}
\end{corollary}

\begin{proposition} 
  \label{MeasurVersClos}
  If $X$ is a random set, then its pointwise closure
  $\cl X(\omega)$, $\omega\in\Omega$, is a random closed set, and
  $\Lnot(\cl X,\cF)=\cl_0\Lnot(X,\cF)$.
\end{proposition}
\begin{proof}
  Since the probability space is complete and the graph of $X$ is
  measurable in the product space, the projection theorem yields that
  $\{X\cap G\neq\emptyset\}\in\cF$ for any open set $G$. Finally,
  note that $X$ hits any open set $G$ if and only if
  $\cl X$ hits $G$. Thus, $\cl X$ is Effros
  measurable and so is a random closed set.

  The inclusion $\cl_0 \Lnot(X,\cF)\subseteq
  \Lnot(\cl X,\cF)$ obviously holds. Since $\cl_0
  \Lnot(X,\cF)$ is decomposable, there exists a random closed set
  $Y$ such that $\cl_0 \Lnot(X,\cF)=\Lnot(Y,\cF)$. Since
  $\Lnot(X,\cF)\subseteq \Lnot(Y,\cF)$, we have
  $X\subseteq Y$ a.s. Therefore, $\cl X\subseteq Y$
  a.s. and the conclusion follows.
\end{proof}

The following result is well known for random closed sets as a
\emph{Castaing representation}, see e.g. \cite[Th.~1.2.3]{mo1}.

\begin{proposition}
  \label{CountRepresentation}
  If $X$ is a non-empty random set, then there exists a countable family
  $\{\xi_i, i\geq1\}$ of measurable selections of $X$ such that
  \begin{displaymath}
    \cl X=\cl\{\xi_i, i\geq1\}\quad \text{a.s.}
  \end{displaymath}
\end{proposition}
\begin{proof}
  It suffices to repeat the part of the proof of
  Theorem~\ref{ReprDecompL0} with $\Xi=\Lnot(\cl X,\cF)$ and
  observe that
  \begin{displaymath}
    a_i\eqdef\inf_{\eta\in
      \Lnot(X,\cF)}\E\|\eta-x_i\|\wedge 1=\inf_{\eta\in
      \Xi}\E\|\eta-x_i\|\wedge 1,\quad i\geq1.
  \end{displaymath}
  Indeed, by Proposition~\ref{MeasurVersClos}, $\Xi=\cl_0
  \Lnot(X,\cF)$.
\end{proof}

Proposition~\ref{CountRepresentation} yields that the norm 
\begin{displaymath}
  \|X\|=\sup\{\|x\|:\; x\in X\}
  =\sup\{\|\xi\|:\; \xi\in\Lnot(X,\cF)\}
\end{displaymath}
and the Hausdorff distance between any two random sets are random
variables with values in $[0,\infty]$.

The following result establishes the existence of a measurable version
for any closed-valued mapping.

\begin{proposition} 
  \label{MeasurVersClos1}
  For any closed set-valued mapping $X(\omega)$, $\omega\in\Omega$,
  there exists a random closed set $Y$ (called the measurable version
  of $X$) such that $\Lnot(X,\cF)=\Lnot(Y,\cF)$. If $X$ is convex
  (respectively, a cone), then $Y$ is also convex (respectively, a
  cone).
\end{proposition}
\begin{proof}
  Assume that $\Lnot(X,\cF)$ is  non-empty, otherwise, the
  statement is evident with empty $Y$. Since
  $\Lnot(X,\cF)$ is closed and decomposable,
  Theorem~\ref{ReprDecompL0} ensures the existence of $Y$
  that satisfies the required conditions. 
\end{proof}

If $X_i$, $i\in I$, is an uncountable family of random sets,
then Proposition~\ref{MeasurVersClos1} makes it possible to define
measurable versions of the closure of their union or intersection. 

For $A,B\subset\XX$, define their pointwise sum as
\begin{displaymath}
  A+B\eqdef\{x+y:\; x\in A,y\in B\}\,.
\end{displaymath}
The same definition applies to the sum of subsets of
$\Lnot[p](\XX,\cF)$.
Note that the sum of two closed sets is not necessarily closed, unless
at least one summand is compact. 
If $X$ and $Y$ are two random closed sets, then the closure of $X+Y$
is a random closed set too, see \cite{mo1}.

\subsection{Infinite decomposability}
\label{sec:infin-decomp-1}

\begin{definition} 
  A family $\Xi\subseteq \Lnot(\XX,\cF)$ is said to be
  \emph{infinitely $\cH$-decomposable} if
  \begin{displaymath}
    \sum_n \xi_n\one_{A_n}\in \Xi
  \end{displaymath}
  for all sequences $\{\xi_n, n\geq1\}$ from $\Xi$ and all
  $\cH$-measurable partitions $\{A_n,n\geq1\}$ of $\Omega$.
\end{definition}

Infinitely $\cF$-decomposable subsets of $\Lnot(\R^d,\cF)$ are called
$\sigma$-stable in \cite{cher:kup:vog14}.  Taking a partition that
consists of two sets and letting all other sets be empty, it is
immediate that the infinite decomposability implies the
decomposability.  Observe that an infinitely decomposable family $\Xi$
is not necessarily closed in $\Lnot$, e.g., $\Xi=\Lnot(X,\cF)$ with a
non-closed $X$.  It is easy to see that if $\Xi$ is infinitely
decomposable, then its closure in $\Lnot$ is also infinitely
decomposable.

In Euclidean space, the sum of an open set $G$ and another set $M$
equals the sum of $G$ and the closure of $M$. The following result
shows that an analogue of this for subsets of $\Lnot(\XX,\cF)$ holds
under the infinite decomposability assumption.

\begin{proposition}
  \label{prop-open+decomposable} 
  Let $\Xi$ be an infinitely $\cF$-decomposable subset of
  $\Lnot(\XX,\cF)$, and let $X$ be an $\cF$-measurable
  a.s. non-empty random open set. Then
  \begin{displaymath}
    \Lnot(X,\cF)+\Xi=\Lnot(X,\cF)+\cl_0\Xi.
  \end{displaymath}
\end{proposition}
\begin{proof}  
  Consider $\gamma\in \Lnot(X,\cF)$ and $\xi\in \cl_0\Xi$, so
  that $\xi_n\to\xi$ a.s. for $\xi_n\in\Xi$, $n\geq1$. By a measurable
  selection argument, there exists an $\alpha\in \Lnot((0,\infty),\cF)$
  such that the ball of radius $\alpha$ centred at $\gamma$ is a
  subset of $X$ a.s. Let us define, up to a null set,
  \begin{displaymath}
    k(\omega)\eqdef\inf\{n:~\|\xi(\omega)-\xi_n(\omega)\|\le
    \alpha(\omega)\}, \quad \omega\in \Omega.
  \end{displaymath}
  Since the mapping $\omega\mapsto k(\omega)$ is $\cF$-measurable,
  \begin{displaymath}
    \hat \xi(\omega)\eqdef\xi_{k(\omega)}(\omega)
    =\sum_{j=1}^\infty \xi_{j}1_{k(\omega)=j},
  \end{displaymath}
  is also $\cF$-measurable and belongs to $\Xi$ by the infinite
  decomposability assumption. Since $\|\hat\xi-\xi\|\leq\alpha$ a.s.,
  \begin{displaymath}
    \xi+\gamma=(\xi+\gamma-\hat \xi)+\hat \xi\in
    \Lnot(X,\cF)+\Xi. \qedhere
  \end{displaymath}
\end{proof}

\begin{corollary} 
  \label{limit-inf-decomposable} 
  Let $\Xi$ be an infinitely $\cF$-decomposable subset of
  $\Lnot(\XX,\cF)$. For every $\gamma\in \cl_0\Xi$ and
  $\alpha\in \Lnot((0,\infty),\cF)$, there exists a $\xi\in \Xi$ such
  that $\|\gamma-\xi\|\le \alpha$ a.s.
\end{corollary}
\begin{proof}
  Apply Proposition~\ref{prop-open+decomposable} with $X$ being the
  open ball of radius $\alpha/2$ centred at zero. 
\end{proof}

\section{Random convex sets}
\label{sec:random-convex-sets}

\subsection{Support function}
\label{sec:supp-funct-cond}

Let $\XX^*$ be the dual space to $\XX$ with the 
pairing $\langle u,x\rangle$ for $x\in\XX$ and $u\in\XX^*$.
The space $\XX^*$ is equipped with the $\sigma(\XX^*,\XX)$-topology
(see \cite[Sec.~5.14]{alip:bor06}) and the corresponding Borel
$\sigma$-algebra, so that $\zeta$ is a random element in $\XX^*$ if
$\langle\zeta,x\rangle$ is a random variable for all $x\in\XX$.  Let
$\XX^*_0$ be a countable total subset of $\XX^*$ (which always
exists). The separability of $\XX$ ensures that the
$\sigma(\XX^*,\XX)$-topology is metrisable, and the corresponding
metric space is complete separable, see \cite{dan:sim79} and 
\cite[Th.~7.8.3]{ma95}. 

The \emph{support function} of a random set $X$ in $\XX$ is defined by
\begin{displaymath}
  h_X(u)=\sup\{\langle u,x\rangle:\; x\in X\}\,,
  \quad u\in\XX^*\,.
\end{displaymath}
The support function of the empty set is set to be $-\infty$. It is
easy to see that the support function does not discern between $X$ and
its closed convex hull. The support function is a lower semicontinuous
sublinear function of $u$; if $X$ is weakly compact, then the support
function is $\sigma(\XX^*,\XX)$-continuous, see
\cite[Th.~7.52]{alip:bor06}. Recall that all topologies consistent
with the pairing have the same lower semicontinuous sublinear
functions.

If $X$ is $p$-integrably bounded, that is $\|X\|\in\Lnot[p](\R,\cF)$
for $p\in[1,\infty]$, then $h_X(\zeta)\in \Lnot[1](\R,\cF)$ for
$\zeta\in\Lnot[q](\XX^*,\cF)$ with $p^{-1}+q^{-1}=1$.  If $X$ is not
bounded, then the support function may take infinite values, even with
probability one for each given $u\in\XX^*$, e.g., if $X$ is a line in
$\XX=\R^2$ with a uniformly distributed direction. This fact calls for
letting the argument of $h_X$ be random. The following result is well
known for deterministic arguments of the support function; it refers
to the definition of the essential supremum from \ref{sec:appendix}.

\begin{lemma}
  \label{lem:h-esssup}
  For every $\zeta\in \Lnot(\XX^*,\cF)$ and a random
  closed convex set $X$, $h_X(\zeta)$ is a random
  variable in $[-\infty,\infty]$, and 
   \begin{displaymath}
     h_X(\zeta)=\esssup[\cF] \{\langle \zeta,\xi\rangle:\; \xi\in
    \Lnot(X,\cF)\}
    \quad \text{a.s.}
  \end{displaymath}
  if $X$ is a.s. non-empty.
\end{lemma}
\begin{proof}
  Since $\{X=\emptyset\}\in\cF$, it is possible to assume that $X$
  is a.s. non-empty.  Taking a Castaing representation
  $X=\cl\{\xi_i, i\geq1\}$, we confirm that $h_X(\zeta)=\sup_i
  \langle\zeta,\xi_i\rangle$ is $\cF$-measurable.
  It is immediate that $h_X(\zeta)\geq \langle \zeta,\xi\rangle$ for
  all $\xi\in\Lnot(X,\cF)$, so that
  \begin{displaymath}
    h_X(\zeta)\geq \esssup[\cF]
    \{\langle \zeta,\xi\rangle:\; \xi\in \Lnot(X,\cF)\}.
  \end{displaymath}
  Assume that $X$ is a.s.\ bounded, so that $|h_X(\zeta)|<\infty$
  a.s.  For any $\eps>0$, the random closed set
  \begin{math}
    X\cap \{x:\; \langle \zeta,x\rangle\geq h_X(\zeta)-\eps\}
  \end{math}
  is a.s. non-empty and so possesses a selection $\eta$. Then 
  \begin{displaymath}
    \esssup[\cF]\{\langle \zeta,\xi\rangle:\; \xi\in
    \Lnot(X,\cF)\}
    \geq \langle \zeta,\eta\rangle\geq h_X(\zeta)-\eps\,.
  \end{displaymath}
  Letting $\eps\downarrow 0$ yields that
  \begin{equation}
    \label{eq:1}
    h_X(\zeta)=\esssup[\cF] \{\langle \zeta,\xi\rangle:\; \xi\in
    \Lnot(X,\cF)\}\qquad \text{a.s.}
  \end{equation}
  For a general closed set $X$, $h_X(\zeta)$ is the limit of
  $h_{X^n}(\zeta)$ as $n\to \infty$, where $X^n$ is the
  intersection of $X$ with the centred ball of radius $n$. Since
  \eqref{eq:1} holds for $X=X^n$ and $X^n\subseteq X$ a.s.,
  \begin{displaymath}
    \langle \zeta,\xi_n\rangle\leq \esssup[\cF] \{\langle
    \zeta,\xi\rangle:\; \xi\in \Lnot(X,\cF)\}\qquad \text{a.s.}
  \end{displaymath}
  for all $\xi_n\in \Lnot(X^n,\cF)$. Hence, 
  \begin{multline*}
    \esssup[\cF] \{\langle
    \zeta,\xi_n\rangle:\; \xi_n\in \Lnot(X^n,\cF)\}\\
    \leq \esssup[\cF]
    \{\langle \zeta,\xi\rangle:\; \xi\in \Lnot(X,\cF)\}\qquad  \text{a.s.}
  \end{multline*}
  Therefore, 
  \begin{displaymath}
    h_X(\zeta)=\lim_{n\to\infty} h_{X^n}(\zeta)\le
    \esssup[\cF] \{\langle \zeta,\xi\rangle:\; \xi\in \Lnot(X,\cF)\},
  \end{displaymath}
  and the conclusion follows.
\end{proof}

\begin{remark}
  \label{SuppFunctionDirected}
  For $\xi_1,\xi_2\in\Lnot(X,\cF)$ and $\zeta\in \Lnot(\XX^*,\cF)$,
  define
  \begin{displaymath}
    \xi\eqdef\xi_11_{\langle \zeta,\xi_1\rangle>\langle \zeta,\xi_2\rangle}
    +\xi_21_{\langle \zeta,\xi_1\rangle\le\langle \zeta,\xi_2\rangle}.
  \end{displaymath}
  Then $\xi\in \Lnot(X,\cF)$ and $\langle \zeta,\xi\rangle=\langle
  \zeta,\xi_1\rangle\vee \langle \zeta,\xi_2\rangle$. Therefore, the
  family $\{\langle \zeta,\xi\rangle:\; \xi\in \Lnot(X,\cF)\}$ is
  directed upward, so that $\langle \zeta,\xi_n\rangle \uparrow
  h_X(\zeta)$, where $\{\xi_n, n\geq1\}\subset\Lnot(X,\cF)$.
\end{remark}

\subsection{Polar sets}
\label{sec:polar-sets}

The \emph{polar} set to a random set $X$ is defined by
\begin{displaymath}
  X^o=\{u\in\XX^*:\; h_X(u)\leq 1\}.
\end{displaymath}
The polar to $X$ is a convex $\sigma(\XX^*,\XX)$-closed set, which
coincides with the polar to the closed convex hull of $X$.

\begin{lemma} 
  \label{lemma:polar}
  If $X$ is a random set in $\XX$, then its polar $X^o$ is a random
  $\sigma(\XX^*,\XX)$-closed convex set.
\end{lemma}
\begin{proof}
  By Proposition~\ref{CountRepresentation}, $\cl X$ admits a
  Castaing representation $\{\xi_i,i\geq1\}$. Then
  \begin{displaymath}
    \Gr X^o=\bigcap_{i\geq1}\{(\omega,u)\in\Omega\times\XX^*:\;
    \langle u,\xi_i(\omega)\rangle\leq 1\}\in \cF\otimes \cB(\XX^*),
  \end{displaymath}
  and it remains to note that $X^o$ is closed for all $\omega$ and to
  note that $\XX^*$ is Polish in the $\sigma(\XX^*,\XX)$-topology. 
\end{proof}

The following result is a variant of the well-known fact saying
that each convex closed set equals the intersection of at most a
countable number of half-spaces. In the setting of random convex sets,
these half-spaces become random and are determined by selections of
$X^o$. 

\begin{theorem}
  \label{thr:countable}
  Each almost surely non-empty random closed convex set $X$ in a
  separable Banach space is obtained as the intersection of an at most 
  countable number of random half-spaces, that is, there exists a
  countable set $\{\zeta_n,n\geq1\}\subset\Lnot(\XX^*,\cF)$ such that
  \begin{equation}
    \label{eq:cr}
    X=\bigcap_{n\geq1} \{x\in \XX:\;
    \langle \zeta_n,x\rangle\leq h_X(\zeta_n)\}.
  \end{equation}
  If $X$ is weakly compact, then it is possible to let $\zeta_n$ be
  deterministic from a countable total set. 
\end{theorem}
\begin{proof}
  Assume first that $X$ almost surely contains the origin, and let
  \begin{displaymath}
    X^o=\cl \{\zeta_n,n\geq1\}
  \end{displaymath}
  be a Castaing representation of $X^o$, which is graph
  measurable by Lemma~\ref{lemma:polar}. Since $X$ is convex, it is
  also weakly closed, and the bipolar theorem yields that
  $X=(X^o)^o$. The second polar does not make a difference between the
  set $\{\zeta_n,n\geq1\}$ and its closure, whence $X$ is the polar set to
  $\{\zeta_n,n\geq1\}$, i.e.
  \begin{equation}
    \label{eq:9}
    X=\bigcap_{n\geq1} \{x\in\XX:\; \langle \zeta_n,x\rangle \leq
    1\}. 
  \end{equation}
  Denote by $\tilde{X}$ the right-hand side of \eqref{eq:cr}.  Since
  $h_X(\zeta_n)\leq1$, the right-hand side of \eqref{eq:9} is a
  superset of $\tilde{X}$. It remains to note that $X$ is a subset of
  $\{x\in \XX:\; \langle \zeta_n,x\rangle\leq h_X(\zeta_n)\}$ for all
  $n$ and so is a subset of $\tilde{X}$. 

  If $X$ does not necessarily contain the origin, consider $Y=X-\xi$
  for any $\xi\in\Lnot(X,\cF)$, so that 
  \begin{align*}
    Y&=\bigcap_{n\geq1} \{y\in \XX:\;
    \langle \zeta_n,y\rangle\leq h_Y(\zeta_n)\}\\
    &=\bigcap_{n\geq1} \{y\in \XX:\;
    \langle \zeta_n,y+\xi\rangle\leq h_X(\zeta_n)\}.
  \end{align*}
  It remains to note that $x\in X$ if and only if $x-\xi\in Y$.
  
  If $X$ is weakly compact, then the support function is
  $\sigma(\XX^*,\XX)$-continuous on $\XX^*$, and $h_X(u)$ is a
  finite random variable for each $u\in\XX^*$. Consequently, $X$ equals
  the intersection of half-spaces $\{x:\; \langle u,x\rangle\leq
  h_X(u)\}$ for all $u\in\XX^*_0$.
\end{proof}

\begin{corollary}
  \label{cor:all-zeta}
  Each almost surely non-empty random closed convex set $X$ in a
  separable Banach space satisfies
  \begin{equation}
    \label{eq:12}
    X=\bigcap_{\zeta\in\Lnot(\XX^*,\cF)} \{x\in \XX:\;
    \langle \zeta,x\rangle\leq h_X(\zeta)\}.
  \end{equation}
\end{corollary}
\begin{proof}
  It suffices to note that $X$ is a subset of the right-hand side of
  \eqref{eq:12}, and the uncountable intersection is a subset of the
  right-hand side of \eqref{eq:9}. 
\end{proof}

\begin{corollary}
  \label{cor:domination}
  If $X$ and $Y$ are two random convex closed sets and
  $h_Y(\zeta)\leq h_X(\zeta)$ a.s. for each
  $\zeta\in\Lnot(\XX^*,\cF)$, then  $Y\subset X$ a.s.
\end{corollary}
\begin{proof}
  Consider the representation of $X$ given by \eqref{eq:9}. Then 
  \begin{displaymath}
    X\supset 
    \bigcap_{n\geq1}\{x\in\XX:\; \langle \zeta_n,x\rangle\leq
    h_Y(\zeta_n)\}
    \supset Y,
  \end{displaymath}
  where the latter inclusion follows from \eqref{eq:12}. 
\end{proof}

Unless the random sets are weakly compact, it does not suffice to
consider deterministic $\zeta$ in Corollary~\ref{cor:domination}.

\subsection{Epigraphs}
\label{sec:epigraphs}

The \emph{epigraph} of a function $f:\XX^*\mapsto[-\infty,\infty]$ is
defined as
\begin{displaymath}
  \epi f=\{(u,t)\in\XX^*\times\R:\; f(u)\leq t\}.
\end{displaymath}
The epigraphs of support functions can be characterised as subsets of
$\XX^*\times\R$ closed in the product of the
$\sigma(\XX^*,\XX)$-topology and the Euclidean topology on $\R$, and which
are convex cones that, with each element $(u,t)$, contain $(u,s)$ for all
$s\geq t$. We denote the family of such subsets by $\cE$. The
following result characterises epigraphs of random closed sets in
$\XX$ and is interesting on its own. Its version for Euclidean spaces
is known, see \cite[Prop.~5.3.6]{mo1}.

\begin{theorem}
  \label{thr:epi}
  A closed convex set-valued mapping $X$ in $\XX$ is $\cF$-measurable if and
  only if $\epi h_X$ is an $\cF$-measurable random closed set with values in $\cE$. 
\end{theorem}
\begin{proof}
  \textsl{Necessity.}  Consider a Castaing representation
  $X=\cl\{\xi_i, i\geq1\}$ of an $\cF$-measurable random closed convex
  set $X$.  Then
  \begin{displaymath}
    h_X(u)=\sup\{\langle x,u\rangle:\; x\in X\}
    =\sup_{i\geq1} \langle \xi_i, u \rangle, \quad u\in \XX^*.
  \end{displaymath}
  Therefore, the graph of $\epi h_X$ is given by
  \begin{equation}
    \Gr  (\epi h_X) =\bigcap_{i\ge 1}\{ (\omega,u,t)\in
    \Omega\times\XX^*\times \R :~t\ge \langle \xi_i(\omega), u \rangle\}
    \label{graphEpi-hX}
  \end{equation}
  Finally, note that the mapping $u\mapsto \langle \xi_i(\omega), u
  \rangle$ is measurable with respect to the product of $\cF$ and
  $\cB(\XX^*)$.

  \textsl{Sufficiency.} Let $Y$ be a random closed set with values in
  $\cE$. Then $Y$ is the epigraph of a lower semicontinuous
  sublinear function
  \begin{displaymath}
    h(u)=\inf\{t:\; (u,t)\in Y\}. 
  \end{displaymath}
  Thus, $h$ is the support function of a set-valued map $X$ with
  closed convex values. It remains to show that $X$ is
  $\cF$-measurable. 

  Let $\{(\zeta_i,t_i),i\geq1\}$ be a Castaing representation of
  $Y$.  Define an $\cF$-measurable random closed set by letting
  \begin{displaymath}
    Z=\bigcap_{n\ge1} \{x\in\XX:\; \langle\zeta_n,x\rangle\leq t_n\}. 
  \end{displaymath}
  Let $(\zeta,t)$ be a selection of $Y$, that is $h(\zeta)\leq t$
  a.s., and assume that $t=h(\zeta)$. 
  By Lemma~\ref{ApproxSelector}, $(\zeta,t)$ is the a.s. limit of
  a sequence $(\zeta'_m,t'_m)$, where $(\zeta'_m,t'_m)$ are obtained
  as combinations of the members of a Castaing representation of
  $Y$. It is easy to see that 
  \begin{displaymath}
    Z\subset \{x\in\XX:\; \langle\zeta'_m,x\rangle\leq t'_m\},\quad m\geq1,
  \end{displaymath}
  whence $h_Z(\zeta'_m)\leq t'_m$ for all $m\geq1$. Note that
  $\zeta'_m\to\zeta$ in the norm topology on $\XX^*$, whence also in
  $\sigma(\XX^*,\XX)$. Passing to the limits and using the lower
  semicontinuity of the support function $h_Z$ yields that
  $h_Z(\zeta)\leq h(\zeta)$. By Corollary~\ref{cor:domination},
  $Z\subset X$. The other inclusion is obvious.
\end{proof}

\section{Conditional core}
\label{sec:conditional-core}

\subsection{Existence}
\label{sec:existence-results}

The following concept is related to the measurable versions of random
closed sets considered in Proposition~\ref{MeasurVersClos1}. 

\begin{definition} 
  Let $X$ be any set-valued mapping. The conditional core $\cm(X|\cH)$
  of $X$ (also called $\cH$-core) is the largest $\cH$-measurable
  random set $X'$ such that $X'\subseteq X$ a.s.
\end{definition}

The following result relates the conditional core to the family of
$\cH$-measurable selections of $X$.

\begin{lemma}
  \label{H-meas-represent-m(X|H)exists} 
  If $\cm(X|\cH)$ exists and is almost surely non-empty, then
  \begin{equation}
    \label{eq:6}
    \Lnot(X,\cH)=\Lnot(\cm(X|\cH),\cH),
  \end{equation}
  in particular $\cm(X,\cH)$ is a.s. non-empty if and only if
  $\Lnot(X,\cH)\ne \emptyset$.
\end{lemma}
\begin{proof}
  In order to show the non-trivial inclusion, consider $\gamma\in
  \Lnot(X,\cH)$. The random set
  $X'(\omega)\eqdef\{\gamma(\omega)\}$ is $\cH$-measurable and
  satisfies $X'\subseteq X$ a.s. It follows that
  $X'\subseteq \cm(X|\cH)$ a.s., so that
  $\gamma\in\cm(X|\cH)$ a.s.
\end{proof}

The existence of the $\cH$-core is the issue of the existence of the
\emph{largest} $\cH$-measurable set $X'\subset X$. It does not prevent
$\cm(X|\cH)$ from being empty.
If $\cH$ is the trivial $\sigma$-algebra, then $\cm(X|\cH)$ is the set
of all points $x\in\XX$ such that $x\in X(\omega)$ almost surely. Such
points are called \emph{fixed points} of a random set and it is
obvious that the set of fixed points may be empty.

\begin{lemma}
  \label{Exist-closed-case} 
  If $X$ is a random closed set, then $\cm(X|\cH)$ exists and is a
  random closed set, which is a.s. convex (respectively, is a cone) if
  $X$ is a.s. convex (respectively, is a cone).
\end{lemma}
\begin{proof}
  We first consider the case where $\Lnot(X,\cH)\ne \emptyset$.
  By Theorem~\ref{ReprDecompL0}, $\Lnot(X,\cH)=\Lnot(Y,\cH)$ for an
  $\cH$-measurable random closed set $Y$. Moreover,
  $Y=\cl\{\xi_n,n\geq1\}$ a.s. for $\xi_n\in \Lnot(X,\cH)$,
  $n\geq1$. Since $X$ is closed,
  $Y\subseteq X$ a.s. Since any $\cH$-measurable random set
  $Z\subseteq X$ satisfies $\Lnot(Z,\cH)\subseteq
  \Lnot(Y,\cH)$, we have $Z\subseteq X$ a.s., so that
  $Y=\cm(X|\cH)$.
  If $X$ is convex, then $\Lnot(X,\cH)=\Lnot(Y,\cH)$ is convex, whence
  $Y$ is a random convex set by Theorem~\ref{ReprDecompL0}.
  
  Let us now consider the case $\Lnot(X,\cH)=\emptyset$. Define 
  \begin{displaymath}
    \cI=\{H\in\cH:\; \Lnot(X,\cH\cap H)\ne \emptyset\},
  \end{displaymath}
  where $\Lnot(X,\cH\cap H)$ designates the $\cH\cap H$-measurable
  selections of $X(\omega)$, $\omega\in H$, measurable with respect to
  the trace of $\cH$ on $H$. Observe that $\cI\ne \emptyset$ if and
  only if there exists a closed $\cH$-measurable subset $Z_{\cH}$ of
  $X$ such that $\P\{Z_{\cH}\ne \emptyset\}>0$. If $\cI= \emptyset$,
  we let $\cm(X|\cH)=\emptyset$. Otherwise, note that $H_1,H_2\in
  \cI$ implies that $H_1\cup H_2\in \cI$. Hence 
  \begin{displaymath}
    \zeta=\esssup\{1_H:\; H\in \cI\}=1_{H^*},
  \end{displaymath}
  where $H_n\uparrow H^*$ for $H_n\in \cI$, $n\geq1$. In particular,
  $H^*\in\cI$.  By the result for non-empty cores, we have
  $\Lnot(X,\cH\cap H^*)=\Lnot(X_{\cH},\cH\cap H^*)$ where $X_{\cH}$ is
  a closed $\cH\cap H^*$-measurable subset which is non-empty on
  $H^*$. This is the largest $\cH\cap H^*$-measurable closed subset of
  $X$.
  
  Define $\cm(X|\cH)$ by $\cm(X|\cH)(\omega)=X_{\cH}(\omega)$ if
  $\omega\in H^*$ and $\emptyset$ otherwise. Consider a closed
  $\cH$-measurable subset $Z$ of $X$. The inclusion $Z\subseteq
  \cm(X|\cH)$ is trivial on $\{\omega:~Z(\omega)= \emptyset\}\in
  \cH$. The complement $H$ of this latter set belongs to $\cI$ as soon
  as $\P(H)>0$, using a measurable selection argument. Therefore, we
  may modify $H$ on a non-null set and suppose that $H\subseteq
  H^*$. We deduce by construction of $X_{\cH}$ that $Z\subseteq
  X_{\cH}$ on $H^*$ while this inclusion is trivial on the complement
  $\{\omega:~Z(\omega)= \emptyset\}$. Thus, $\cm(X|\cH)$ is the
  largest closed $\cH$-measurable subset of $X$.
\end{proof}

\begin{lemma}
  \label{lemma:lp-core}
  If $\Lnot[p](X,\cH)\neq\emptyset$ for some $p\in[1,\infty]$ and a random
  closed set $X$, then $\Lnot[p](\cm(X|\cH),\cH)=\Lnot[p](X,\cH)$.
\end{lemma}
\begin{proof}
  By the condition, $\cm(X|\cH)$ admits a $p$-integrable selection,
  and so has a Castaing representation consisting of $p$-integrable
  selections, see \cite{hia:ume77} and \cite{mo1}.
\end{proof}

\begin{example}
  \label{example:core-finite-omega}
  If $\cH$ is generated by a partition $\{B_1,\dots,B_m\}$ of a finite
  probability space, then $Y=\cm(X|\cH)$ is obtained by letting
  $Y(\omega)=\cap_{\omega'\in B_i} X(\omega')$ if $\omega\in B_i$,
  $i=1,\dots,m$.
\end{example}

\subsection{Sublinearity of the conditional core}
\label{sec:subadd-prop}

\begin{lemma}
  \label{LemmCondminProduct}
  Let $X$ be an $\cF$-measurable random set. 
  \begin{enumerate}[(i)]
  \item $\lambda X$ is an $\cF$-measurable random set for any
    $\lambda\in \Lnot(\R,\cF)$.
  \item If 
    $\lambda\in \Lnot(\R,\cH)$, then
    \begin{equation}
      \label{eq:scale}
      \cm(\lambda X|\cH)=\lambda\cm(X|\cH).
    \end{equation}
  \item If $X$ is a random closed set and $\cH'$ is a
    sub-$\sigma$-algebra of $\cH$, then
    \begin{displaymath}
      \cm(\cm(X|\cH)|\cH')=\cm(X|\cH').
    \end{displaymath}
  \end{enumerate}
\end{lemma}
\begin{proof}
  (i) Since 
  \begin{displaymath}
    \Gr(\lambda X)=(\{\lambda=0\}\times\{0\})
    \cup(\Gr(\lambda X)\cap(\{\lambda\neq 0\}\times\XX)),
  \end{displaymath}
  it suffices to assume that $\lambda\neq 0$ a.s.  The measurability
  of $\lambda X$ is immediate, since the map
  $\phi:(\omega,x)\mapsto (\omega,\lambda^{-1} x)$ is measurable and
  $\Gr(\lambda X)=\phi^{-1}(\Gr X)$.
  
  (ii)
  Observe that $\lambda\cm(X|\cH)$ is $\cH$-measurable if
  $\lambda\in \Lnot(\R,\cH)$ and $\lambda\cm(X|\cH)\subseteq
  \lambda X$. Suppose that $X'\subseteq \lambda X$ is
  $\cH$-measurable. Then
  \begin{displaymath}
    X''=\lambda^{-1}X'\one_{\lambda\ne 0}+X\one_{\lambda=0}\subseteq X
  \end{displaymath}
  is $\cH$-measurable. Therefore, $X''\subseteq
  \cm(X|\cH)$, so that $X'\subseteq
  \lambda\cm(X|\cH)$. Thus, $\cm(\lambda X|\cH)$ exits and
  \eqref{eq:scale} holds.

  (iii) By Lemma~\ref{H-meas-represent-m(X|H)exists},
  \begin{align*}
    \Lnot(\cm(\cm(X|\cH)|\cH'),\cH')&=\Lnot(\cm(X|\cH),\cH)\cap
    \Lnot(\XX,\cH')\\
    &=\Lnot(X,\cH')=\Lnot(\cm(X|\cH'),\cH'). \qedhere
  \end{align*}
\end{proof}

The following result establishes that the conditional core is
subadditive for the reverse inclusion ordering; together with
Lemma~\ref{LemmCondminProduct}(ii), they mean that the conditional core
is a set-valued conditional \emph{sublinear expectation}.

\begin{lemma} 
  \label{LemmCondminSum} 
  Let $X$ and $Y$ be set-valued mappings.
  Then
  \begin{equation}
    \label{eq:2}
    \cm(X|\cH)+ \cm(Y|\cH)\subseteq  \cm(X+Y|\cH).
  \end{equation}
\end{lemma}
\begin{proof}
  The inclusion is trivial on the $\cH$-measurable subset where one of
  the conditional cores on the left-hand side of \eqref{eq:2} is
  empty, since then the sum is also empty.  If $\gamma'\in
  \Lnot(\cm(X|\cH),\cH)$ and $\gamma''\in \Lnot(\cm(Y|\cH),\cH)$, then
  \eqref{eq:6} yields that
  \begin{displaymath}
    \gamma'+\gamma''\in
    \Lnot(X+Y,\cH)=\Lnot(\cm(X+Y|\cH),\cH).
    \qedhere
  \end{displaymath}
\end{proof}

\begin{example}
  \label{ex:H-summand}
  Let $\cH$ be trivial, and let $X$ be a line in the plane passing
  through the origin with a random direction, so that
  $\cm(X|\cH)=\{0\}$. If $Y$ is a deterministic centred ball, then the
  set of fixed points of $X+Y$ may be strictly larger than
  $\{0\}+Y=Y$.
\end{example}

While Example~\ref{ex:H-summand} shows that \eqref{eq:2} does not
necessarily turn into the equality if one of the summands is
$\cH$-measurable, the equality holds if one of the summands is an 
$\cH$-measurable \emph{singleton}.

\begin{lemma}
  \label{lem:sing-sum}
  If $X$ is a random closed set and $\eta\in\Lnot(\XX,\cH)$, then
  \begin{displaymath}
    \cm(X+\{\eta\}|\cH)=\cm(X|\cH)+\eta.
  \end{displaymath}
\end{lemma}
\begin{proof}
  Consider $\xi+\eta\in\Lnot(\cm(X+\{\eta\}|\cH),\cH)$. Then $\xi$ is
  also $\cH$-measurable, so that $\xi\in\Lnot(\cm(X|\cH),\cH)$. The
  opposite inclusion follows from Lemma~\ref{LemmCondminSum}. 
\end{proof}

\begin{lemma} 
  Let $X$ be a random closed set such that there exists a
  $\gamma\in\Lnot(X,\cH)$.  Let $X^n$ be the intersection of $X$ with
  the closed ball of radius $n\geq1$ centred at $\gamma$. Then
  \begin{displaymath}
    \cm(X|\cH)=\cl\bigcup_{n\geq1} \cm(X^n|\cH).
  \end{displaymath}
\end{lemma}
\begin{proof}
  Since $\cm(X^n|\cH)\subseteq X^n\subseteq X$, we have 
  $\cl\bigcup_n\cm(X^n|\cH)\subseteq X$, whence
  \begin{displaymath}
    \cl\bigcup_n\cm(X^n|\cH)\subseteq\cm(X|\cH).
  \end{displaymath}
  Reciprocally, consider a selection $\xi$ of $\cm(X|\cH)\subseteq X$
  (otherwise, $\cm(X|\cH)=\emptyset$ and the inclusion is
  trivial). Then
  \begin{displaymath}
    \xi^n\eqdef\xi\one_{\|\xi-\gamma\|\le n}+\gamma \one_{\|\xi-\gamma\|> n}\in
    \Lnot(X^n,\cH),
  \end{displaymath}
  so that $\xi^n\in \cm(X^n|\cH)$ a.s. Letting $n\to \infty$ finishes
  the proof. 
\end{proof}

The strong upper limit $\slimsup X_n$ of a sequence
$\{X_n,n\geq1\}$ of random sets is defined as the set of limits for
each almost surely strongly convergent sequence
$\xi_{n_k}\in\Lnot(X_{n_k},\cF)$, $k\geq1$.

\begin{proposition}
  \label{prop:slimsup}
  If $\slimsup X_n\subset X$ a.s. for a sequence of random
  closed sets $\{X_n,n\geq1\}$ and a random closed set $X$, then 
  \begin{equation}
    \label{eq:7}
    \slimsup \cm(X_n|\cH)\subset \cm(X|\cH) \quad \text{a.s.}
  \end{equation}
\end{proposition}
\begin{proof}
  If $\gamma_{n_k}\in \cm(X_{n_k}|\cH)$ and $\gamma_{n_k}\to \gamma$,
  then $\gamma$ is $\cH$-measurable, and almost surely belongs to $X$,
  whence $\gamma$ is a selection of $\cm(X|\cH)$. 
\end{proof}

It should be noted that the reverse inclusion in~\eqref{eq:7} does not
hold, e.g., if $X_n=\{\xi_n\}$ with a non-$\cH$-measurable $\xi_n$ such
that $\xi_n$ a.s. converges to an $\cH$-measurable $\xi$.

\subsection{Essential infimum of the support function}
\label{sec:essent-infim-supp}

It is possible to relate the conditional core of random closed convex
sets to the conditional essential infimum of their support
functions. 

\begin{theorem} 
  \label{thr:sf-cm} 
  Let $X$ be an a.s. non-empty random closed convex set.  Then
  \begin{equation}
    \label{eq:support-inequality}
    h_{\cm(X|\cH)}(\zeta)\le \essinf h_X(\zeta),
    \quad \zeta\in \Lnot(\XX^*,\cH),
  \end{equation}
  and 
  \begin{equation}
    \label{eq:cm-from-zeta-1}
    \cm(X|\cH)=\bigcap_{\zeta\in\Lnot(\XX^*,\cF)}\Big\{x\in\XX:\;
    \essinf \big(h_X(\zeta)-\langle \zeta,x\rangle\big)\geq0\}.
  \end{equation}
  If $\{\zeta_n,n\geq1\}$ is a sequence from \eqref{eq:cr}, then 
  \begin{equation}
    \label{eq:cm-from-zeta}
    \cm(X|\cH)=\bigcap_{n\geq1}\Big\{x\in\XX:\;
    \essinf \big(h_X(\zeta_n)-\langle \zeta_n,x\rangle\big)\geq0\}.
  \end{equation}
  If $X$ is weakly compact, then 
  \begin{equation}
    \label{eq:zeta-intersection}
    \cm(X|\cH)=\bigcap_{u\in \XX^*_0}\{x\in \XX:\;
    \langle u,x\rangle\leq \essinf h_X(u)\},
  \end{equation}
  and the intersection can be taken over all $u\in\XX^*$. 
\end{theorem}
\begin{proof}
  Without loss of generality assume that $\cm(X|\cH)\neq\emptyset$.
  Fix any $\zeta\in \Lnot(\XX^*,\cH)$.  By Lemma~\ref{lem:h-esssup},
  \begin{align*}
    h_{\cm(X|\cH)}(\zeta)&=\esssup\{\langle \zeta,\xi\rangle:\;
    \xi\in \Lnot(\cm(X|\cH),\cH)\}\\
    &=\esssup \{\langle \zeta,\xi\rangle:\;
    \xi\in \Lnot(X,\cH)\}\,.
  \end{align*}
  Moreover, $\cm(X|\cH)\subseteq X \subseteq \{x:~\langle x,
  \zeta\rangle\le h_X(\zeta)\}$.  Thus, $\langle \zeta,\xi\rangle\le
  h_X(\zeta)$ for $\xi\in \Lnot(X,\cH)$. Since $\langle
  \zeta,\xi\rangle$ is $\cH$-measurable, $\langle \zeta,\xi\rangle\le
  \essinf h_X(\zeta)$, so that \eqref{eq:support-inequality} holds. 

  Each selection $\xi$ of $\cm(X|\cH)$ is also a selection of the
  right-hand side of \eqref{eq:cm-from-zeta}, since
  $h_X(\zeta_n)-\langle\zeta_n,\xi\rangle\geq0$ a.s. for all $n$. The function
  \begin{displaymath}
    \eta(x)=\essinf \big(h_X(\zeta_n)-\langle\zeta_n,x\rangle\big)
    =\essinf h_{X-x}(\zeta_n)
  \end{displaymath}
  is Lipschitz in $x$ and so is continuous. Since $\eta(x)$ is
  $\cH$-measurable for all $x$, the function is jointly measurable in
  $(\omega,x)$. Thus, each set on the right-hand side is an
  $\cH$-measurable random closed set, whence the intersection is also
  a random closed set, and the equality follows from the maximality of
  the conditional core. The right-hand side of
  \eqref{eq:cm-from-zeta-1} is a subset of the right-hand side of
  \eqref{eq:cm-from-zeta} and contains $\cm(X|\cH)$, whence
  \eqref{eq:cm-from-zeta-1} holds. 

  By choosing $\zeta=u$ in \eqref{eq:support-inequality}, we see that
  the left-hand side of \eqref{eq:zeta-intersection} is a subset of
  the right-hand one with the intersection taken over all $u\in\XX^*$.
  The right-hand side of \eqref{eq:zeta-intersection} is an
  $\cH$-measurable random closed set denoted by $\tilde{X}$. It suffices
  to assume that $\tilde{X}$ is a.s. non-empty and consider its
  $\cH$-measurable selection $\gamma$. For all $u\in\XX^*_0$,
  \begin{displaymath}
    \langle u,\gamma\rangle \leq \essinf h_X(u) \leq h_X(u),
  \end{displaymath}
  whence $\gamma\in X$ a.s., and $\tilde{X}\subset \cm(X|\cH)$ a.s.
\end{proof}

Equation \eqref{eq:cm-from-zeta-1} may be thought as the dual
representation of the conditional core.

\begin{example}
  The inequality in \eqref{eq:support-inequality} can be strict. Let
  $X$ be a line in the plane $\XX=\R^2$ passing through the origin
  with the normal vector $\zeta$ having a non-atomic distribution and
  such that $\langle x,\zeta\rangle$ is not $\cH$-measurable for ant
  $x\neq0$. Furthermore, assume that $\cH$ contains all null-events
  from $\cF$. Then the only $\cH$-measurable selection of $X$ is the
  origin, so that the left-hand side of \eqref{eq:support-inequality}
  vanishes. For each deterministic (and so $\cH$-measurable)
  non-vanishing $u$, we have $h_X(u)=\infty$ a.s., so that the 
  right-hand side of \eqref{eq:support-inequality} is infinite.  Still
  \eqref{eq:cm-from-zeta} holds with $\zeta_1=\zeta$ and
  $\zeta_2=-\zeta$. Indeed, then $\langle \zeta,x\rangle=0$ a.s., which
  is only possible for $x=0$.
\end{example}

Note that $\essinf h_X(u)$ is not necessarily subadditive as function
of $u$ and so may fail to be a support function. Recall that the
conjugate of a function $f:\XX\mapsto(-\infty,\infty]$ is defined by 
\begin{displaymath}
  f^o(u)=\sup_{x\in\XX}\Big(\langle u,x\rangle- f(x)\Big),\quad u\in\XX^*,
\end{displaymath}
and the biconjugate of $f$ is the conjugate of $f^o:\XX^*\mapsto(-\infty,\infty]$. 

\begin{proposition}
  \label{prop:conjugate}
  Let $X$ be a random convex closed set. Then the support function of
  $\cm(X|\cH)$ is the largest $\cH$-measurable lower
  semicontinuous sublinear function $h:\XX^*\mapsto(-\infty,\infty]$
  such that \eqref{eq:support-inequality} holds. If $X$ is weakly 
  compact, then the support function of $\cm(X|\cH)$ is the
  biconjugate function to $\essinf h_X(u)$, $u\in\XX^*$.
\end{proposition}
\begin{proof}
  By Theorem~\ref{thr:sf-cm}, \eqref{eq:support-inequality} holds.
  If $h$ is the largest lower semicontinuous function such
  that \eqref{eq:support-inequality} holds, then it corresponds to 
  an $\cH$-measurable random closed set $Y$. The
  conclusion follows from the definition of the conditional core as
  the largest $\cH$-measurable subset of $X$. 

  If $X$ is weakly compact, then it is possible to let the argument
  of the support function be non-random. The largest sublinear
  function dominated by  that  $\essinf h_X(u)$, $u\in\XX^*$, is its
  biconjugate. 
\end{proof}

\section{Conditional convex hull}
\label{sec:cond-conv-hull}

\subsection{Existence and construction}
\label{sec:definition}

\begin{definition}
  If $X$ is a random set, then its conditional convex hull
  $\CM(X|\cH)$ is the smallest $\cH$-measurable random convex closed
  set which contains $X$.
\end{definition}

\begin{theorem}
  \label{thr:exist-CM}
  If $X$ is a random set, then $\CM(X|\cH)$ exists, and 
  \begin{equation}
    \label{eq:13}
    h_{\CM(X|\cH)}(\zeta)=\esssup h_X(\zeta),\quad
    \zeta\in\Lnot(\XX^*,\cH). 
  \end{equation}
\end{theorem}
\begin{proof}
  Without loss of generality, assume that $X$ is closed and convex
  (with possibly empty values). Then the epigraph $\epi h_X$ is a
  closed convex cone in $\XX^*\times\R$, so that $\cm(\epi h_X|\cH)$
  exists and is a convex cone by Lemma~\ref{Exist-closed-case}. By
  Theorem~\ref{thr:epi}, $\cm(\epi h_X|\cH)$ is the epigraph of the
  support function of an $\cH$-measurable random closed convex set
  $Z$. The support function of $Z$ is the smallest $\cH$-measurable
  support function that dominates $h_X$ and so $Z=\CM(X|\cH)$ is the smallest
  $\cH$-measurable random closed convex set containing $X$.

  The left-hand side of \eqref{eq:13} is greater than or equal to the
  right-hand one. Since 
  \begin{align*}
    \esssup h_X(u+v)&\leq \esssup(h_X(u)+h_X(v))\\
    &\leq \esssup h_X(u)+\esssup h_X(v)
  \end{align*}
  for all $u,v\in\XX^*$, the essential supremum retains the
  subadditivity property and so is a support function. Thus, the
  right-hand side is a support function of an $\cH$-measurable random
  closed convex set, which is a subset of $\CM(X|\cH)$ by
  Corollary~\ref{cor:domination}. The definition of the conditional
  convex hull yields the equality. 
\end{proof}

By Corollary~\ref{cor:all-zeta}, $\CM(X|\cH)$ equals the intersection
of random half-spaces $\{x:\; \langle \zeta,x\rangle\leq \esssup
h_X(\zeta)\}$ over $\zeta\in\Lnot(\XX^*,\cH)$. If $\CM(X|\cH)$ is
a.s. weakly compact, then it is possible to let $\zeta$ run over
deterministic elements from $\XX^*_0$.

\begin{proposition}
  \label{prop:union}
  If $X^n$ is the intersection of a random closed set $X$ with the
  centred ball of radius $n$, then 
  \begin{equation}
    \label{eq:CM-Union}
    \CM(X|\cH)=\cl\bigcup_n\CM(X^n|\cH).
  \end{equation}
\end{proposition}
\begin{proof}
  If $X\subseteq Y$ for an $\cH$-measurable random closed convex set $Y$,
  then $X^n\subseteq Y$. Thus, $\CM(X^n|\cH) \subseteq Y$ for all $n$,
  whence $\CM(X|\cH) \subseteq Y$.
\end{proof}

\subsection{Superadditivity of the conditional convex hull}
\label{sec:super-cond-conv}

\begin{proposition} 
  \label{coroCondMaxSum} 
  If $X$ and $Y$ are random sets, then 
  \begin{equation}
    \label{eq:3}
    \CM(X+Y|\cH)\subset \cl\Big(\CM(X|\cH)+\CM(Y|\cH)\Big).
  \end{equation}
  If $Y$ is $\cH$-measurable and convex, then 
  \begin{displaymath}
    \CM(X+Y|\cH)=\cl(\CM(X|\cH)+Y).
  \end{displaymath}
\end{proposition}
\begin{proof}
  For \eqref{eq:3}, it suffices to note that the right-hand side is a
  convex closed set that is $\cH$-measurable and contains
  $X+Y$. The second statement follows from \eqref{eq:13}, since
  \begin{align*}
    h_{\CM(X+Y|\cH)}(u)&=\esssup h_{X+Y}(u),\\
    &=\esssup (h_X(u))+h_Y(u)=h_{\CM(X|\cH)}(u)+h_Y(u). \qedhere
  \end{align*}
\end{proof}

Proposition~\ref{coroCondMaxSum} together with the rather obvious
homogeneity property imply that the conditional convex hull is a
set-valued conditional \emph{superlinear expectation}. 

\subsection{Duality between the core and the convex hull}
\label{sec:duality-between-core}

The following result establishes a relationship between the
conditional convex hull and the conditional core assuming that the
random set $X$ contains at least one $\cH$-measurable selection.

\begin{proposition}
  \label{prop:CM-gamma}
  If $\Lnot(X,\cH)\neq\emptyset$, then 
  \begin{displaymath}
    \CM(X-\gamma|\cH)=(\cm((X-\gamma)^o|\cH))^o
  \end{displaymath}
  for any $\gamma\in \Lnot(X,\cH)$.
\end{proposition}
\begin{proof}
  It suffices to assume that $X$ is a random convex closed set and
  $0\in X$ a.s., so let $\gamma=0$ a.s. Then
  $(\cm(X^o|\cH))^o$ contains $X$ a.s. If $Y$ is an
  $\cH$-measurable random convex closed set which contains $X$, then
  $Y^o\subset X^o$ a.s., whence $Y^0\subset \cm(X^o|\cH)$ a.s., and
  $Y$ contains the polar to the latter set.
\end{proof}

A similar duality relationship holds for epigraphs, namely,
\begin{align*}
  \CM(\epi h_X|\cH)&=\epi h_{\cm(X|\cH)},\\
  \cm(\epi h_X|\cH)&=\epi h_{\CM(X|\cH)}.
\end{align*}

\begin{remark}
  Consider a filtration $(\cF_t)_{t=0,\dots,T}$ on
  $(\Omega,\cF,\P)$. An adapted sequence $(X_t)_{t=0,\dots,T}$ of
  random closed set is said to be 
  a \emph{maxingale} if $X_s=\CM(X_t|\cF_s)$
  for all $s\leq t$ the same holds for the
  conditional convex hull. 
  If $X$ is a random closed set, then $X_t=\CM(X|\cF_t)$, $t=0,\dots,T$,
  is a maxingale. Random sets $X_t=(-\infty,\xi_t]$ form a maxingale
  if and only if the sequence $(\xi_t)_{t=0,\dots,T}$ of random
  variables is a maxingale in the sense of \cite{bar:car:jen03}. 
  A similar concept applies for the conditional core.
  If the conditional core (or convex hull) is replaced by the
  expectation, one recovers the concept of a set-valued martingale,
  see \cite{hia:ume77} and \cite[Sec.~5.1]{mo1}.
\end{remark}

\section{Conditional expectation}
\label{sec:expect-cond-expect}

\subsection{Integrable random sets}
\label{sec:integrable-case}

\begin{definition}[see \cite{hia:ume77}]
  Let $X$ be an \emph{integrable} random closed set, that is
  $\Lnot[1](X,\cF)\neq\emptyset$. The conditional expectation
  $\E(X|\cH)$ with respect to a $\sigma$-algebra $\cH\subset\cF$ is
  the random closed set such that 
  \begin{equation}
    \label{eq:4}
    \Lnot[1](\E(X|\cH),\cH) =\cl_1\{\E(\xi|\cH):\;\xi\in\Lnot[1](X,\cF)\}.
  \end{equation}
\end{definition}

The following result shows that it is possible to take the
$\Lnot$-closure in \eqref{eq:4}.

\begin{lemma} 
  \label{lemma:lo-closure}
  $\Lnot(\E(X|\cH),\cH)$ coincides with the
  $\Lnot$-closure of the set
  $\{\E(\xi|\cH):\xi\in\Lnot[1](X,\cF)\}$.
\end{lemma}
\begin{proof} 
  By definition, $\tilde
  \Xi=\{\E(\xi|\cH):\;\xi\in\Lnot[1](X,\cF)\}$ is a subset of 
  $\Lnot(\E(X|\cH),\cH)$, which is closed in $\Lnot$ since
  $\E(X|\cH)$ is a.s. closed. Therefore, $\cl_0 \tilde \Xi\subseteq
  \Lnot(\E(X|\cH),\cH)$. By Proposition~\ref{CountRepresentation},
  the random set $\E(X|\cH)$ admits a Castaing representation
  $\{\xi_i, i\geq1\}$, where $\xi_i\in\cl_1\tilde \Xi$ for all
  $i\geq1$. Then $\{\xi_i,i\geq1\}\subseteq \cl_0 \tilde \Xi$, so that
  $\Lnot(\E(X|\cH),\cH)\subseteq \cl_0 \tilde \Xi$ by
  Lemma~\ref{ApproxSelector}.
\end{proof}

\subsection{Generalised conditional expectation of random sets}
\label{sec:gener-cond-expect-1}

The following definition relies on the concept of the generalised
expectation discussed in \ref{sec:gener-cond-expect}. 

\begin{definition}
  Let $X$ be a random closed set and let $\cH$ be a
  sub-$\sigma$-algebra of $\cF$ such that $\Lnot[1_\cH](X,\cF)\ne
  \emptyset$. The generalised conditional expectation $\E^g(X|\cH)$
  is the $\cH$-measurable random closed set such that
  \begin{equation}
    \label{eq:8}
    \Lnot(\E^g(X|\cH),\cH)
    =\cl_0\{\E^g(\xi|\cH):\;\xi\in\Lnot[1_\cH](X,\cF)\}.
  \end{equation}
\end{definition}

The existence of the generalised conditional expectation follows from
Corollary~\ref{cor:H-dec}, since the family on the right-hand side of
\eqref{eq:8} is $\cH$-decomposable. 

\begin{lemma} 
  \label{LemmEquCondExp-G} 
  If $X$ is an \emph{integrable} random closed set, then
  $\E(X|\cH)=\E^g(X|\cH)$.
\end{lemma}
\begin{proof} 
  To show the non-trivial inclusion, consider $\xi\in
  \Lnot[1_\cH](X,\cF)$, so that $\E^g(\xi|\cH)=\sum_{i=1}^\infty
  \E(\xi \one_{A_i}|\cH)\one_{A_i}$ for an $\cH$-measurable partition
  $\{A_i, i\geq1\}$, and $\xi\one_{A_i}\in \Lnot[1](\XX,\cF)$ for
  all $i\ge 1$. If $\gamma\in \Lnot[1](X,\cF)$, then
  \begin{displaymath}
    \E^g(\xi|\cH)=\lim_{n\to\infty} \sum_{i=1}^n \Big[\E(\xi \one_{A_i}|\cH)\one_{A_i}
    +\E(\gamma|\cH)1_{\Omega\setminus \cup_{i\le n}A_i}\Big]
    \quad \text{a.s.}
  \end{displaymath}
  Since $\xi\one_{A_i}$ and $\gamma$ are integrable, the sum under the
  limit belongs to $\E(X|\cH)$. Therefore, $\E^g(\xi|\cH)\in\E(X|\cH)$
  a.s. Since $\E(X|\cH)$ is a random closed set, the family
  $\Lnot(\E(X|\cH),\cH)$ is closed in $\Lnot$ by
  Lemma~\ref{lemma:lo-closure}. Thus, $\E^g(X|\cH)\subseteq \E(X|\cH)$
  a.s.
\end{proof}

\begin{lemma}
  \label{Link-GExp-Exp}
  Let $X$ be a random closed set such that $\Lnot[1_\cH](X,\cF)\ne
  \emptyset$. Then, for every $\xi\in \Lnot[1_\cH](X,\cF)$,
  $X-\xi$ is an \emph{integrable} random closed set, and
  \begin{displaymath}
    \E^g(X|\cH)=\E(X-\xi|\cH)+\E^g(\xi|\cH)
    \quad \text{a.s.}
  \end{displaymath}
\end{lemma}
\begin{proof} 
  Since $X-\xi$ is integrable, $\E(X-\xi|\cH)=\E^g(X-\xi|\cH)$
  by Lemma~\ref{LemmEquCondExp-G}.  Then
  \begin{displaymath}
    \E^g(\eta|\cH)=\E^g(\eta-\xi|\cH)+\E^g(\xi|\cH)\in
    \E^g(X-\xi|\cH)+\E^g(\xi|\cH)\quad \text{a.s.}
  \end{displaymath}
  for all $\eta\in \Lnot[1_\cH](X,\cF)$.  Therefore,
  $\E^g(X|\cH)\subseteq \E(X-\xi|\cH)+\E^g(\xi|\cH)$ a.s. The
  reverse inclusion follows from 
  \begin{displaymath}
    \E^g(X-\xi|\cH)\subseteq \E^g(X|\cH)-\E^g(\xi|\cH). \qedhere
  \end{displaymath}
\end{proof}

\begin{corollary} 
  \label{Exp(X|H)=X} 
  If $X$ is an $\cH$-measurable a.s. non-empty random closed convex
  set, then $\E^g(X|\cH)=X$ a.s.
\end{corollary}

It is well known, see \cite{hia:ume77} and \cite{mo1}, that if $X$ is
a.s. convex and integrable, then $h_{\E(X|\cH)}(u)=\E(h_X(u)|\cH)$
a.s.  for all $u\in\XX^*$, see \cite[Th.~2.1.47]{mo1}.  The following
result is a generalisation of this fact for possibly random arguments of the
support function. 

\begin{lemma}
  \label{lemm:sup-cond}
  If $\Lnot[1_\cH](X,\cF)\neq\emptyset$, then 
  \begin{displaymath}
    \E^g(h_X(\zeta)|\cH)=h_{\E^g(X|\cH)}(\zeta),\quad \zeta\in\Lnot(\XX^*,\cH).
  \end{displaymath}
\end{lemma}
\begin{proof}
  By passing from $X$ to $X-\gamma$ for $\gamma\in
  \Lnot[1_\cH](X,\cF)\neq\emptyset$, it is possible to assume that $X$
  contains the origin with probability one and work with the conventional
  conditional expectation.
  
  Each $\eta\in\Lnot(\E(X|\cH),\cH)$ is the almost sure limit of
  $\E(\xi_n|\cH)$ for $\xi_n\in\Lnot[1_\cH](X,\cF)$, $n\geq1$. 
  Then 
  \begin{displaymath}
    \langle\zeta,\eta\rangle=\lim \E(\langle\zeta,\xi_n\rangle|\cH)
    \leq \E(h_X(\zeta)|\cH).
  \end{displaymath}
  Thus, $h_{\E^g(X|\cH)}(\zeta)\leq \E^g(h_X(\zeta)|\cH)$ a.s.

  In the other direction, fix $\eps>0$ and $c>0$ and let 
  \begin{displaymath}
    Y=\{x\in\XX:\; \langle\zeta,x\rangle\geq h_X(\zeta)-\eps\}
    \cup\{x\in\XX:\; \langle \zeta,x\rangle\geq c\}.
  \end{displaymath}
  Then $X\cap Y$ is an almost surely non-empty random closed set,
  which possesses a selection $\xi$ such that 
  \begin{displaymath}
    \langle \zeta,\xi\rangle\geq \min(h_X(\zeta)-\eps,c).
  \end{displaymath}
  Passing to conditional expectations yields that 
  \begin{displaymath}
    \E(\min(h_X(\zeta)-\eps,c)|\cH)\leq
    \E(\langle\zeta,\xi\rangle|\cH)
    \leq h_{\E(X|\cH)}(\zeta).
  \end{displaymath}
  Since the support function of $X$ is non-negative, letting
  $c\uparrow\infty$ and $\eps\downarrow0$ concludes the proof.
\end{proof}

Note that Lemma~\ref{lemm:sup-cond} holds for the conventional
conditional expectations if $X$ is integrable and $\zeta\in
\Lnot[\infty](\XX^*,\cH)$, that is, the (strong) norm of $\zeta$ is
essentially bounded.

\begin{lemma}
  \label{LemmaUnionExpectation}
  Let $X$ be a random closed set such that
  $\Lnot[1_\cH](X,\cF)\neq\emptyset$, and let $X^n=X\cap B_n$,
  $n\geq1$. Then
  \begin{displaymath}
    \E^g(X|\cH)=\cl\bigcup_n\E^g(X^n|\cH)\quad \text{a.s.}
  \end{displaymath}
\end{lemma}
\begin{proof}
  By passing from $X$ to $X-\gamma$ for any $\gamma\in
  \Lnot[1_\cH](X,\cF)$, it is possible to assume that $0\in X$
  a.s., so that $X$ is integrable. 

  Denote the right-hand side by $Y$.  Note that $Y\subseteq
  \E(X|\cH)$.  To confirm the reverse inclusion, let $\xi=
  \E(\eta|\cH)$ for $\eta\in \Lnot[1](X,\cF)$. Then $\xi$ is the
  limit of $\E(\eta_n|\cH)$ in $\Lnot[1]$, where
  $\eta_n\eqdef\eta\one_{|\eta|\le n}\in \Lnot[1](X^n,\cF)$. Since
  $\E(\eta_n|\cH)\in \E(X^n|\cH)$ a.s., $\xi\in Y$ a.s. and the
  conclusion follows.
\end{proof}

\subsection{Sandwich theorem}
\label{sec:sandwich-theorem}\bigskip

Now we show that the (generalised) conditional expectation is
sandwiched between the conditional core and the conditional convex
hull. 

\begin{proposition}
  \label{prop:cond-exp-cm-CM}
  If $X$ is an a.s. non-empty random closed set, then
  \begin{equation}
    \label{eq:inclusion-cond-exp}
    \cm(X|\cH)\subseteq \E^g(X|\cH)\subseteq\CM(X|\cH)
    \quad \text{a.s.}
  \end{equation}
\end{proposition}
\begin{proof}
  The first inclusion is trivial, unless
  $\Lnot(X,\cH)\ne \emptyset$. Then each $\cH$-measurable selection
  $\xi$ of $\cm(X|\cH)$ satisfies $\xi=\E^g(\xi|\cH)$, whence the
  first inclusion holds. 

  For the second inclusion, note that $X\subset \CM(X|\cH)$, and
  \begin{displaymath}
    \E^g(\CM(X|\cH)|\cH)=\CM(X|\cH)
  \end{displaymath}
  by Corollary~\ref{Exp(X|H)=X}. 
\end{proof}

If $0\in X$ a.s., then Proposition~\ref{prop:CM-gamma} yields that
\begin{displaymath}
  \cm(X|\cH)\subseteq \E^g(X|\cH)\subseteq (\cm(X^o|\cH))^o, 
\end{displaymath}
whence
\begin{displaymath}
  \cm(X|\cH)\subset \Big(\E^g(X|\cH)\cap (\E^g(X^o|\cH))^o\Big). 
\end{displaymath}

Consider the family $\cQ$ of all probability measures $\Q$ absolutely
continuous with respect to $\P$.  The following result can be viewed
as an analogue of the representation of superlinear and sublinear
functions as suprema and infima of linear functions.

\begin{theorem}
  \label{thr-Comparison-m-M-E}
  Let $X$ be a random closed convex set.
  Then $\CM(X|\cH)$ (respectively, $\cm(X|\cH)$) is the smallest
  (respectively, largest) $\cH$-measurable random closed convex set
  a.s. containing $\E^g_\Q(X|\cH)$ for all $\Q\in\cQ$ such that the
  generalised conditional expectation exists.
\end{theorem}
\begin{proof}
  Consider the family $\{\zeta_n,n\geq1\}$ which yields $\CM(X|\cH)$
  by an analogue of \eqref{eq:cr}. By \eqref{eq:13} and
  Theorem~\ref{essCharact},
  \begin{displaymath}
    \E^g_{\Q_{nm}}(h_X(\zeta_n)|\cH)\uparrow 
    h_{\CM(X|\cH)}(\zeta_n)\quad \text{a.s. as }\; m\to\infty
  \end{displaymath}
  for a sequence $\{\Q_{nm},m\geq1\}\subset \cQ$. By Lemma \ref{lemm:sup-cond}, 
  \begin{align*}
    h_{\CM(X|\cH)}(\zeta_n)
    =h_{\bigcup_m \E^g_{\Q_{nm}}(X|\cH)}(\zeta_n)
    \le h_{\bigcup_{n,m} \E^g_{\Q_{nm}}(X|\cH)}(\zeta_n).
  \end{align*}
  In view of \eqref{eq:cr},
  \begin{displaymath}
    \CM(X|\cH)\subseteq\cl\co \bigcup_{n,m\geq1} \E^g_{\Q_{nm}}(X|\cH). 
  \end{displaymath}
  Proposition \ref{prop:cond-exp-cm-CM} yields the reverse inclusion, so
  that the equality holds. The union can be, equivalently, taken over
  all $\Q\in\cQ$, letting the generalised conditional expectation to
  be empty if $X$ does not contain any selection which does not admit
  the generalised conditional expectation under $\Q$.

  If $Z$ is an $\cH$-measurable subset of $\E^g_\Q(X|\cH)$ for all
  $\Q\in\cQ$, then
  \begin{displaymath}
    h_Z(\zeta)=\E^g_\Q(h_Z(\zeta)|\cH)\leq
    \E^g_\Q(h_X(\zeta)|\cH).
  \end{displaymath}
  By Theorem~\ref{essCharact}, $h_Z(\zeta)\leq \essinf h_X(\zeta)$ 
  for all $\zeta\in\Lnot(\XX^*,\cH)$.
  By Proposition~\ref{prop:conjugate}, 
  \begin{displaymath}
    h_Z(\zeta)\leq h_{\cm(X|\cH)}(\zeta).
  \end{displaymath}
  By Corollary~\ref{cor:domination},
  $\cm(X|\cH)\supset Z$.
\end{proof}

Following the idea of Theorem~\ref{thr-Comparison-m-M-E}, it is
possible to come up with a general way of constructing non-linear
set-valued expectations. If $\cM$ is a sub-family of $\cQ$, then a measurable
version of 
\begin{displaymath}
  \bigcap_{\Q\in\cM} \E_\Q(X|\cH)
\end{displaymath}
is a set-valued sublinear conditional expectation that satisfies
\eqref{eq:2} and a measurable version of
\begin{displaymath}
  \cl \bigcup_{\Q\in\cM} \E_\Q(X|\cH)
\end{displaymath}
satisfies \eqref{eq:3}. 

\begin{example}
  \label{ex:sv-avar}
  Assume that $\cM$ consists of all probability measures $\Q$ such
  that $d\Q/d\P\leq \alpha^{-1}$ for some $\alpha\in(0,1)$. Then the
  above formula provide set-valued sub- and superlinear analogues of
  the conditional Average Value-at-Risk, which is a well-known risk
  measure, see \cite{delb12,foel:sch04}.
\end{example}

\section{Polar sets and random cones}
\label{sec:random-cones}

If $X$ is a \emph{convex cone} in $\XX$, then $X^*=-X^o$ is called the
positive dual cone to $X$, so that
\begin{displaymath}
  X^*=\{u\in \XX^*:\; \langle u,x\rangle\ge 0\;\forall x\in X\}.
\end{displaymath}

\begin{proposition}
  \label{dual-CondMin-max}\label{lemma:exp-cone}
  Let $\bK$ be a random convex closed cone in $\XX$. Then both
  $\cm(\bK|\cH)$ and $\CM(\bK|\cH)$ are closed convex cones,
  $\cm(\bK|\cH)=\CM(\bK^*|\cH)^*$, and $\E(\bK|\cH)=\CM(\bK|\cH)$ a.s.
\end{proposition}
\begin{proof}
  The conical properties of the core and the convex hull are
  obvious. Since $\cm(\bK|\cH)\subseteq \bK$,
  \begin{displaymath}
    \bK^*\subseteq \CM(\bK^*|\cH)\subseteq \cm(\bK|\cH)^*
  \end{displaymath}
  in view of the definition of $\CM(\bK^*|\cH)$. Therefore,
  $\cm(\bK|\cH)\subseteq \CM(\bK^*|\cH)^*$. The opposite inclusion
  follows from 
  \begin{displaymath}
    \CM(\bK^*|\cH)^*\subseteq \cm(\bK|\cH)\subseteq \bK
  \end{displaymath}
  by the definition of the conditional core.

  For the last statement, assume that
  $\gamma\in\Lnot[1](\CM(\bK|\cH),\cH)$ does not belong to
  $\Lnot[1](\E(\bK|\cH),\cH)$. By the Hahn-Banach separation theorem,
  there exist $\eta\in \Lnot[\infty](\XX^*,\cH)$ and $c\in\R$ such
  that
  \begin{displaymath}
    \E\langle \xi,\eta\rangle <c< \E\langle\gamma,\eta\rangle
  \end{displaymath}
  for all $\xi\in \Lnot[1](\E(\bK|\cH),\cH)$. Since
  $\Lnot[1](\E(\bK|\cH),\cH)$ is a cone, we have $c>0$ and
  $-\eta$ belongs to $\Lnot[1](\E(\bK|\cH)^*,\cH)$. Since
  $\cm(\bK|\cH)\subseteq \E(\bK|\cH)$,
  \begin{displaymath}
    \E(\bK|\cH)^*\subseteq \cm(\bK|\cH)^*=\CM(\bK^*|\cH)
  \end{displaymath}
  by Proposition~\ref{dual-CondMin-max}. Therefore, $-\eta\in \bK^*$
  a.s. Since $\gamma\in\bK$ a.s., $\E\langle \gamma,\eta\rangle\le 0$
  in contradiction with $c>0$.
  
  The opposite inclusion follows from
  Theorem~\ref{thr-Comparison-m-M-E}(i).
\end{proof}

Each random convex closed set $X$ in $\XX$ gives rise 
to a random convex cone $Y=\cone(X)$ in $\R_+\times \XX$ given by 
\begin{equation}
  \label{eq:10}
  \cone(X)=\{(t,tx):\; t\geq0,\; x\in X\}. 
\end{equation}
Note that $Y^*$ is not necessarily a subset of $\R_+\times\XX$ and so
cannot be represented by \eqref{eq:10}. 

\begin{proposition}
  \label{prop:perspective}
  If $Y=\cone(X)$ is given by \eqref{eq:10}, then
  $\cm(Y|\cH)=\cone(\cm(X|\cH))$ and, if $\CM(X|\cH)$ is a.s. bounded,
  then also $\CM(Y|\cH)=\cone(\CM(X|\cH))$.
\end{proposition}
\begin{proof}
  By definition, $\cm(Y|\cH)\supset \cone(\cm(X|\cH))$, since the
  latter set is $\cH$-measurable. If $(\xi_0,\xi)$ is an
  $\cH$-measurable selection of $\cm(Y|\cH)$, then $\eta=\xi/\xi_0$
  a.s. belongs to $X$ and is $\cH$-measurable, whence
  $\eta\in\cm(X|\cH)$, and $(\xi_0,\xi)=\xi_0(1,\eta)$. 

  Obviously, $\CM(Y|\cH)\subset \cone(\CM(X|\cH))$. We show that
  $\E(Y|\cH)=\cone(\CM(X|\cH))$. The support function of $Y$ is given
  by 
  \begin{align*}
    h_Y((u_0,u))&=\sup\{tu_0+t\langle u,x\rangle:\; t\geq 0, x\in X\}\\
    &=
    \begin{cases}
      0, & \text{if }\; u_0+h_X(u)\leq 0,\\
      \infty & \text{otherwise}.
    \end{cases}
  \end{align*}
  Thus, $\E h_Y((u_0,u))=0$ if $u_0+h_X(u)\leq 0$ a.s. and is infinite
  otherwise. It suffices to note that $u_0+h_X(u)\leq 0$ a.s. if and
  only if $u_0\leq -\esssup h_X(u)$ and refer to
  Theorem~\ref{thr:exist-CM}.  
\end{proof}

\begin{example}[Random cone in $\R^2$]
  If finance, the random segment $X=[S^b,S^a]\subset\R_+$ models the
  bid-ask spread, and the positive dual cone to $Y=\cone(X)$ is called
  the solvency cone, see
  \cite{kab:saf09}. Proposition~\ref{prop:perspective} shows that 
  $\cm(Y|\cH)$ is the cone generated by $[\esssup S^b,\essinf S^a]$, while
  $\CM(Y|\cH)$ is generated by $[\essinf S^b,\esssup S^a]$. 
\end{example}

\renewcommand{\thesection}{Appendix A}
\renewcommand{\thetheorem}{A.\arabic{theorem}}
\renewcommand{\theequation}{A.\arabic{equation}}

\section{Conditional essential supremum}
\label{sec:appendix}

Let $\Xi\subseteq \Lnot(\R,\cF)$ be a (possibly uncountable) family of
real-valued $\cF$-measurable random variables and let $\cH$ be a
sub-$\sigma$-algebra of $\cF$. The following result is well known, see
e.g. \cite[Appendix~A.5]{foel:sch04} and further refinements in
\cite{bar:car:jen03} and \cite{kab:lep-1-13}. 

\begin{theorem} 
  For any family $\Xi$ of random variables, there exits a unique
  $\hat\xi\in \Lnot((-\infty,+\infty],\cH)$, denoted by $\esssup \Xi$
  and called the $\cH$-conditional supremum of $\Xi$, such that
  $\hat\xi \ge \xi$ a.s. for all $\xi\in\Xi$, and $\eta \ge \xi$
  a.s. for an $\eta\in \Lnot((-\infty,+\infty],\cH)$ and all
  $\xi\in\Xi$ implies $\eta \ge \hat\xi$ a.s.
\end{theorem}

It is easy to 
verify the tower property
\begin{displaymath}
  \esssup[\cH']\esssup[\cH] \Xi
  =\esssup[\cH']\xi
\end{displaymath}
if $\cH'\subseteq \cH$.

Let $\cQ$ be the set of all probability measures $\Q$ absolutely
continuous with respect to $\P$. In the following, $\E_\Q$ designates
the expectation under $\Q$.

\begin{theorem}
  \label{essCharact}
  Let $\Xi\subseteq \Lnot(\R,\cF)$ and let $\cH$ be a
  sub-$\sigma$-algebra of $\cF$. If each $\xi\in\Xi$ is
  a.s. non-negative or its generalised conditional expectation
  $\E^g_\Q(\xi|\cH)$ exists for all $\Q\in \cQ$, then
  \begin{displaymath}
    \esssup \Xi=\esssup[\cF]\{\E^g_\Q(\xi|\cH), \xi\in\Xi, \Q\in \cQ\}\,.    
  \end{displaymath}
  Moreover, if $\Xi=\{\xi\}$ is a singleton, the family
  $\{\E^g_\Q(\xi|\cH), \Q\in \cQ\}$ is directed upward, and there
  exists a sequence $\Q_n\in \cQ$, $n\geq1$, such that
  $\E_{\Q_n}(\xi|\cH)\uparrow\esssup \Xi$ everywhere on $\Omega$.
\end{theorem}
\begin{proof}
  Since $\esssup \Xi\ge \xi$ for all $\xi\in\Xi$ and
  $\esssup \Xi$ is $\cH$-measurable, $\esssup \Xi\ge
  \E^g_\Q(\xi|\cH)$ for all $\Q\in \cQ$ and $\xi\in\Xi$. Therefore, 
  \begin{displaymath}
    \esssup \Xi\geq \esssup[\cF]\{\E_\Q(\xi|\cH),\xi\in\Xi,\Q\in \cQ\}
    =\tilde\gamma\,.  
  \end{displaymath}
  It remains to show that $\tilde\gamma\geq\xi$ a.s. for all $\xi\in
  \Xi$. This is trivial on the set
  $\{\tilde\gamma=\infty\}$. Therefore, we may assume without loss of
  generality that $\tilde\gamma<\infty$ a.s.  Assume that there exist
  a $\xi\in\Xi$ and a non-null set $A\in \cF$ such that
  $\tilde\gamma<\xi$ on $A$. Then $\tilde\gamma \one_{A}+\xi
  \one_{A^c}\le \xi$ and the inequality is strict on $A$. Let
  $d\Q=\alpha \one_{A}d\P$ for $\alpha=\P(A)^{-1}$, so that $\Q\in
  \cQ$. By definition of $\tilde\gamma$,
  \begin{displaymath}
    \E^g_\Q(\xi|\cH)\one_{A}+\xi \one_{A^c}\le \xi,
  \end{displaymath}
  and the inequality is strict on $A$. Taking the
  conditional expectation yields that
  \begin{displaymath}
    \E^g_\Q(\xi|\cH)\E^g_\Q(\one_{A}|\cH)+\E^g_\Q(\xi \one_{A^c}|\cH)
    \le \E_\Q(\xi|\cH),
  \end{displaymath}
  whence
  \begin{displaymath}
    \E^g_\Q(\xi \one_{A^c}|\cH)\le
    \E^g_\Q(\xi|\cH)\E^g_\Q(\one_{A^c}|\cH),
  \end{displaymath}
  and the inequality is strict on $A$. Indeed, the random variables in
  the inequality above take their values in $\R$ by assumption. We
  then obtain a contradiction, since $\Q(A^c)=0$.
  
  To show that the family $\{\E^g_\Q(\xi|\cH), \Q\in \cQ\}$ is directed
  upward, consider $\Q_1, \Q_2\in \cQ$ such that $d\Q_i=\alpha_id\P$,
  $i=1,2$. Define $\Q\in \cQ$ by letting $d\Q=c\alpha d\P$ with
  \begin{displaymath}
    \alpha= \alpha_1\one_{\E_{\Q_1}(\xi|\cH)\ge
      \E_{\Q_2}(\xi|\cH)}+\alpha_2\one_{\E_{\Q_1}(\xi|\cH)<
      \E_{\Q_2}(\xi|\cH)}
  \end{displaymath}
  and $c>0$ chosen such that $\Q(\Omega)=1$. Then 
  \begin{displaymath}
    \E^g_{\Q}(\xi\one_{H})=\E^g_{\P}(\alpha\xi\one_{H})\ge
    \E^g_{\Q}(\E^g_{\Q^i}(\xi|\cH)\one_{H})\quad  \text{a.s.}
  \end{displaymath}
  for every $H\in \cH$, whence $\E^g_\Q(\xi|\cH)\ge
  \E^g_{\Q_i}(\xi|\cH)$ a.s. for $i=1,2$. The conclusion follows.
\end{proof}

Similar definitions and results hold for the conditional essential infimum.

\renewcommand{\thesection}{Appendix B}
\renewcommand{\thetheorem}{B.\arabic{theorem}}
\renewcommand{\theequation}{B.\arabic{equation}}

\section{Generalised conditional expectation}
\label{sec:gener-cond-expect}

\begin{definition}
  \label{GeneralisedCondExp} 
  Let $\cH$ be a sub-$\sigma$-algebra of $\cF$. We say that the
  generalised conditional expectation of $\xi\in \Lnot(\XX,\cF)$
  exists if there exists an $\cH$-measurable partition
  $\{A_i, i\geq1\}$ such that $\xi\one_{A_i}$ is integrable for all
  $i\ge 1$. In this case, we say that $\xi\in \Lnot[1_\cH](\XX,\cF)$
  and let 
  \begin{equation}
    \label{eq:gep}
    \E^g(\xi|\cH)=\sum_{i=1}^\infty \E(\xi \one_{A_i}|\cH)\one_{A_i}.
  \end{equation}
\end{definition}

It is easy to see that the generalised conditional expectation does
not depend on the chosen partition. 

\begin{theorem} 
  \label{EquivalentDef-G-Cond-Exp} 
  We have $\xi\in\Lnot[1_\cH](\XX,\cF)$ if and only if
  $\xi\in\Lnot(\XX,\cF)$ with $\E(\|\xi\||\cH)<\infty$ a.s. For random
  variables $\xi\in \Lnot[1_\cH](\R,\cF)$, we have 
  $\E^g(\xi|\cH)=\E(\xi^+|\cH)-\E(\xi^-|\cH)$, where $\xi^+=\xi\vee 0$
  and $\xi^-=-(\xi\wedge 0)$.
\end{theorem}
\begin{proof}
  If $\E(\|\xi\| |\cH)<\infty$ a.s., define an $\cH$-measurable
  partition by letting
  \begin{displaymath}
    A_n=\{\omega:~\E(\|\xi\| |\cH)\in [n,n+1) \}\in\cH,\quad n\ge 0.
  \end{displaymath}
  Since 
  \begin{displaymath}
    \E(\|\xi\|\one_{A_n})=\E(\E(\|\xi\| |\cH)\one_{A_n})\le
    n+1,
  \end{displaymath}
  we have $\xi\one_{A_n}$ is integrable, and 
  \begin{displaymath}
    \E^g(\xi|\cH)=\sum_n \E(\xi\one_{A_n}|\cH)\one_{A_n}
  \end{displaymath}
  If $\xi$ is a random variable, then $\E(\xi^+|\cH)<\infty$ and
  $\E(\xi^-|\cH)<\infty$ a.s. and 
  \begin{align*}
  \E^g(\xi|\cH)&=\sum_n(\E(\xi^+|\cH)-\E(\xi^-|\cH))\one_{A_n}\\
    &=\E(\xi^+|\cH)-\E(\xi^-|\cH).
  \end{align*}

  Reciprocally, let $\xi\in\Lnot[1_\cH](\XX,\cF)$, i.e.  there exists
  an $\cH$-measurable partition $\{A_i,i\geq1\}$ such that
  $\|\xi\|\one_{A_i}$ is integrable for all $i\ge 1$. Then
  \begin{displaymath}
    \E(\|\xi\| |\cH)=\sum_n\E(\|\xi\|\one_{A_n} |\cH)\one_{A_n}.
  \end{displaymath}
  Since $\|\xi\|\one_{A_n}$ is integrable,
  $\E(\E(\|\xi\|\one_{A_n}|\cH))=\E(\|\xi\|\one_{A_n})<\infty$, so that
  $\E(\|\xi\|\one_{A_n}|\cH)<\infty$ a.s. and $\E(\|\xi\||\cH)<\infty$
  a.s.
\end{proof}

\begin{lemma} 
  \label{Decomposition-rv-CondExp}
  Random element $\xi\in\Lnot(\XX,\cF)$ admits a generalised
  conditional expectation if and only if $\xi=\gamma\tilde\xi$, where
  $\gamma\in\Lnot([1,\infty),\cH)$ and $\tilde\xi$ is integrable.
\end{lemma}
\begin{proof} 
  Let $\xi$ admit a generalised conditional expectation. Define
  $\gamma\eqdef(1+\E(\|\xi\||\cH))$, so that $\E(\|\tilde \xi\| |\cH)\le
  1$. Hence, $0\le \E(\|\tilde \xi| )\le 1$ and $\tilde\xi$ is
  integrable. Reciprocally, if $\xi=\gamma\tilde\xi$ with
  $\gamma\in\Lnot([1,\infty),\cH)$ and integrable $\tilde\xi$,
  then $\xi\one_{A_i}$ is integrable for
  $A_i=\{\gamma\in[i,i+1)\}$, $i\ge 1$.
\end{proof}

\begin{lemma}
  \label{lemma:linear}
  If $\xi\in \Lnot[1_\cH](\XX,\cF)$ and $\zeta\in\Lnot(\XX^*,\cH)$,
  then 
  \begin{displaymath}
    \E^g(\langle\zeta,\xi\rangle|\cH)=\langle\zeta,\E^g(\xi|\cH)\rangle.
  \end{displaymath}
\end{lemma}
\begin{proof}
  If $\{A_n,n\geq1\}$ is the partition from \eqref{eq:gep}, then the
  statement follows by partitioning $\Omega$ with
  $A_{nm}=A_n\cap\{\|\zeta\|\in[m,m+1)\}$ and using the linearity
  property of the conditional expectation. 
\end{proof}

In view of Lemma~\ref{EquivalentDef-G-Cond-Exp}, let
$\Lnot[p_\cH](\XX,\cF)$ with $p\in[1,\infty]$ be the family of
$\xi\in\Lnot(\XX,\cF)$ such that $\E(\|\xi\|^p|\cH)<\infty$ a.s. if
$p\in[1,\infty)$ and $\esssup\|\xi\|<\infty$ a.s. if $p=\infty$.

\begin{lemma}[Dominated convergence]
  \label{Dominated-convergence-theorem} 
  Let $\{\xi_n, n\ge 1\}$ be a sequence from $\Lnot[p_\cH](\XX,\cF)$
  with $p\in[1,\infty)$ which converges a.s. to
  $\xi\in\Lnot(\XX,\cF)$. If $\|\xi_n\|\le\gamma$ a.s. for some
  $\gamma\in\Lnot[p_\cH](\R_+,\cF)$ and all $n$, then $\xi_n\to\xi$ in
  $\Lnot[p_\cH](\XX,\cF)$.
\end{lemma}
\begin{proof} 
  Consider a partition $\{A_i,i\geq1\}$ of elements from $\cH$ such
  that $\gamma\one_{A_i}\in\Lnot[p](\XX,\cF)$ for all $i\ge 1$. Observe
  that $\|\xi\|\le\gamma$ a.s. Applying the conditional dominated
  convergence theorem for integrable random variables, we obtain that
  $\E(\|\xi_n-\xi\|^p\one_{A_i}|\cH)\to 0$ for all $i\ge 1$. Hence,
  $\E(\|\xi_n-\xi\|^p|\cH)\to 0$ as $n\to\infty$.
\end{proof}

\begin{lemma}
  \label{denseL^p} 
  The set $\Lnot[p](\XX,\cF)$ is dense in
  $\Lnot[p_{\cH}](\XX,\cF)$ for all $p\in[1,\infty]$.
\end{lemma}
\begin{proof}
  Let $p\in[1,\infty)$. Consider $\xi\in\Lnot[p_{\cH}](\XX,\cF)$. By
  Lemma~\ref{Decomposition-rv-CondExp}, $\xi=\gamma \tilde\xi$ where
  $\gamma\in \Lnot([1,\infty),\cH)$ and $\tilde
  \xi\in\Lnot[p](\XX,\cF)$. Define
  $\xi_n\eqdef\gamma\one_{\|\gamma\|\le n}\tilde\xi\in
  \Lnot[p](\XX,\cF)$. By Lemma~\ref{Dominated-convergence-theorem},
  $\xi_n\to \xi$ in $\Lnot[p_{\cH}](\XX,\cF)$.  For $p=\infty$, a
  similar argument applies with $\xi_n\eqdef\xi$ if
  $\esssup[\cH]\|\xi\|\le n$ and $\xi_n\eqdef0$ otherwise.
\end{proof}

\begin{lemma}
  \label{Extraction-subsequence} 
  Let $\{\xi_n, n\geq1\}$ be a sequence from $\Lnot[p_{\cH}](\XX,\cF)$
  which converges to $\xi$ in $\Lnot[p_{\cH}](\XX,\cF)$. Then there
  exists a random $\cH$-measurable sequence $\{n_k, k\ge 1\}$ of
  $\cH$-measurable natural numbers, such that
  $\xi_{n_{k}}\in\Lnot[p_{\cH}](\XX,\cF)$ and $\xi_{n_k}\to \xi$ a.s.
\end{lemma}
\begin{proof} 
  Since $\xi_n\to\xi$ in $\Lnot[p_{\cH}](\XX,\cF)$, we deduce that
  $\E(\|\xi_m-\xi\|^p|\cH)\to0$ a.s. as $m\to\infty$. Define the
  $\cH$-measurable sequence $\{n_k, k\ge 1\}$ of natural numbers by
  \begin{align*} 
    n_1&\eqdef\inf\{n:~\E(\|\xi_i-\xi\|^p|\cH)\le 2^{-p},
    \,\text{ for\,all\,}i\ge n\},\\
    n_{k+1}&\eqdef\inf\{n> n_k:~\E(\|\xi_i-\xi\|^p|\cH)\le
    2^{-p(k+1)},
    \,\text{ for\,all\,}i\ge n\}.
  \end{align*}
  Since $\E(\|\xi_{n_{k}}-\xi\|^p|\cH)\le 2^{-pk}$, we deduce that
  $\E(\|\xi_{n_{k}}-\xi\|^p)\le 2^{-pk}$. Therefore,
  $(\xi_{n_{k}}-\xi)\to0$ in $\Lnot[p](\XX,\cF)$ and almost surely
  for a subsequence. Observe that
  \begin{displaymath}
    \E(\|\xi_{n_{k}}\||\cH)=\sum_{j\ge k}\E(\|\xi_{j}\||\cH)\one_{n_{k}=j}
    <\infty,
  \end{displaymath}
  so that $\xi_{n_{k}}\in\Lnot[p_{\cH}](\XX,\cF)$. The conclusion
  follows.
\end{proof}

\end{document}